\documentclass[10pt]{amsart}
\usepackage{a4,amssymb,times,amsmath,fullpage}

\usepackage[dvips,all]{xy}

\def\l{\leftarrow}

\def\Id{\operatorname{id}}

\let\cal\mathcal

\def\cD{{\cal D}}

\def\cF{{\cal F}}

\def\cK{{\cal K}}

\def\cM{{\cal M}}

\def\cP{{\cal P}}

\def\cS{{\cal S}}

\def\cW{{\cal W}}

\let\blb\mathbb

\def \PP{{\blb P}}

\newcommand{\iso}{{\;\stackrel{_\sim}{\to}\;}}

\newcommand{\se}[1]{\begin{equation*}\begin{split}#1\end{split}\end{equation*}}

\newcommand{\onto}{\twoheadrightarrow}

\newcommand{\C}{\mathbb{C}}
\newcommand{\N}{\mathbb{N}}
\newcommand{\Z}{\mathbb{Z}}
\newcommand{\R}{\mathbb{R}}

\newcommand{\Tr}{\textrm{Tr}}

\newcommand{\cycle}{\circlearrowright}

\newcommand{\<}{\langle}
\renewcommand{\>}{\rangle}
\newcommand{\spl}{\mathsf{split}}
\newcommand{\coev}{\mathsf{coev}}
\newcommand{\Aut}{\ensuremath{\mathsf{Aut}}}
\newcommand{\Out}{\ensuremath{\mathsf{Out}}}
\newcommand{\Inn}{\ensuremath{\mathsf{Inn}}}

\newcommand{\Mod}{\ensuremath{\mathsf{Mod}}}
\newcommand{\Bimod}{\ensuremath{\mathsf{Bimod\,}}}
\newcommand{\Span}{\ensuremath{\mathsf{Span}}}

\newcommand{\Ext}{\mathsf{Ext}}

\font\ef = eufb10
\newcommand{\ideal}[1]{\ensuremath{\text{\ef{#1}}}}
\newcommand{\Symm}{\ideal{S}}
\newcommand{\Sym}{\operatorname{Sym}}

\newtheorem{lemma}{Lemma}[section]
\newtheorem{proposition}[lemma]{Proposition}
\newtheorem{theorem}[lemma]{Theorem}
\newtheorem{corollary}[lemma]{Corollary}

\theoremstyle{definition}

\newtheorem{example}[lemma]{Example}
\newtheorem{definition}[lemma]{Definition}

{
\newtheorem{claim}{Claim}

}

\theoremstyle{remark}

\newtheorem{remark}[lemma]{Remark}

\newtheorem{notation}{Notation$\!\!$}

\newcommand{\Mat}{\texttt{Mat}}

\newcommand{\Hom}{\textrm{Hom}}

\newcommand{\Image}{\textrm{Im}}

\newcommand{\SL}{\ensuremath{\mathsf{SL}}}
\newcommand{\id}{\mathbf{1}}

\newcommand{\e}{\varepsilon}
\renewcommand{\t}[1]{\textnormal{#1}}
\numberwithin{equation}{section}

\title{Superpotentials and Higher Order Derivations}

\author{Raf Bocklandt}
\address{Raf Bocklandt\\University of Antwerp\\Middelheimlaan 1\\B-2020 Antwerpen (Belgium)}
\email{rafael.bocklandt@ua.ac.be}
\thanks{The first author is a Postdoctoral Fellow of the Fund for Scientific Research - Flanders (Belgium)}
\author{Travis Schedler}
\address{Travis Schedler\\Department of Mathematics\\ 5734 S University Ave\\ Chicago, IL 60637}
\email{trasched@math.uchicago.edu}
\author{Michael Wemyss}
\address{Michael Wemyss\\Department of Mathematics\\University Walk\\Bristol BS8 1TW (UK)}
\email{M.Wemyss@bristol.ac.uk}

\xyoption{all}

\begin{document}
\begin{abstract}
  We consider algebras defined from quivers with relations that are
  $k$-th order derivations of a superpotential, generalizing results
  of Dubois-Violette to the quiver case.  We give a construction
  compatible with Morita equivalence, and show that many important
  algebras arise in this way, including McKay correspondence algebras
  for $GL_n$ for all $n$, and four-dimensional Sklyanin algebras.
  More generally, we show that any $N$-Koszul, (twisted) Calabi-Yau
  algebra must have a (twisted) superpotential, and construct its
  minimal resolution in terms of derivations of the (twisted)
  superpotential.  This yields an equivalence between $N$-Koszul
  twisted Calabi-Yau algebras $A$ and algebras defined by a
  superpotential $\omega$ such that an associated complex is a
  bimodule resolution of $A$.  Finally, we apply these results to
  give a description of the moduli space of four-dimensional Sklyanin
  algebras using 
  the Weil representation of $\SL_2(\Z/4)$.
\end{abstract}

\maketitle

\section{Introduction}\label{intro}

Let $Q$ be a quiver (with set of edges also denoted by $Q$), and
$\C Q$ its path algebra.  This means that $Q$ is an oriented graph,
and $\C Q$ is the algebra with $\C$-linear basis given by paths in the
graph, with multiplication given by concatenation of paths (setting
$p \cdot q = 0$ if $p$ and $q$ cannot be concatenated).

If $p$ and $q$ are paths we define the partial
derivative of $q$ with respect to $p$ as
\[
 \partial_p q := \begin{cases}
                  r &\text{if $q=pr$,}\\
          0 &\text{otherwise.}
                 \end{cases}
\]
We can extend this operation linearly to get a map
$\partial_p:\C Q \to \C Q$.  Note that if $p=e$ is a trivial path
(i.e. a vertex) then taking the derivative is the same as
multiplication on the left: $\partial_e q=eq$.

Similarly to \cite{dbvfirst, dbv}, we define the
\emph{derivation-quotient algebra} of $\omega\in \C Q$ of order $k$ as
the path algebra modulo the derivatives of $\omega$ by paths with
length $k$:
\begin{equation} \label{cdokdfn}
\cD(\omega,k) := \C Q/\<\partial_p \omega; |p|=k\>.
\end{equation}
  We are particularly interested in $\omega$ which
are super-cyclically symmetric, i.e., are a sum of elements of the
form
\begin{equation} \label{scsymm}
\sum_{i=1}^n (-1)^{(i-1)(n+1)} a_i a_{i+1} \cdots a_n a_1 a_2 \cdots a_{i-1}, \quad a_i \in Q.
\end{equation}
Such $\omega$ will be called \emph{superpotentials}.  We will also consider
\emph{twisted superpotentials}. These are elements in $(\C Q)_n$ that are invariant under the $\C$-linear map
\begin{equation}
a_1\cdots a_n \mapsto (-1)^{n-1}\sigma(a_n)a_1\cdots a_{n-1}.
\end{equation}
where $\sigma$ is a $\C$-algebra automorphism of $\C Q$ which fixes
the trace function
$\displaystyle \C Q \mathop{\twoheadrightarrow}^{\varepsilon} \C^I
\mathop{\rightarrow}^{\Tr} \C$.
(If $\sigma$ is trivial one recovers the notion of a superpotential. We
need twisted superpotentials to address the McKay correspondence for
$GL_n$, as opposed to $\SL_n$.)

For the case $k=1$ in \eqref{cdokdfn}, algebras defined by
superpotentials have been greatly studied.  Examples include
3-dimensional Sklyanin algebras \cite{AS}, algebras coming from the
$3$-dimensional McKay correspondence \cite{GB, CT}, and algebras
derived from exceptional collections on Calabi-Yau varieties
\cite{aspin}.  The fact that all these algebras have a superpotential
can be traced back to a common homological property: the Calabi-Yau
property. In one of its forms, this property states that an algebra is
CY-$n$ if it has a resolution $\cP^\bullet$ as bimodule over itself
that is self-dual:
\begin{equation} \label{cycond}
 \Hom_{A-A}(\cP^\bullet, A\otimes_\C A) \cong \cP^{n-\bullet}.
\end{equation}
Similarly, one has the twisted Calabi-Yau property, where the
resolution is self-dual to a twist of itself by an automorphism $\sigma$ of $A$:
\[
 \Hom_{A-A}(\cP^\bullet, A\otimes_\C A) \cong \cP^{n-\bullet} \otimes_A A^{\sigma},
\]
where $A^{\sigma}$ is the bimodule obtained from $A$ by twisting the
right multiplication by $\sigma$ ($a \cdot x \cdot b = ax\sigma(b)$,
for $x \in A^{\sigma}$, $a,b \in A$). 

It is known that graded
3-dimensional Calabi-Yau algebras always derive from a superpotential
\cite{bocklandt}, i.e., are of the form \eqref{cdokdfn} with $k=1$.
Also, in \cite[Theorem 3.6.4]{GB}, a wide class of Calabi-Yau algebras
of any dimension are shown to arise from a much more general type of
superpotential.

In \cite{dbvfirst,dbv}, in the one-vertex case (working over a field), these
results were generalized to higher-order derivations.  
In particular, \cite[Theorem 11]{dbv} implies that any AS-Gorenstein algebra
over a field which is also Koszul is equal to $\cD(\omega, k)$ for
some $\omega, k$ (more generally, this is shown replacing Koszul with
$N$-Koszul, a generalization to the case of algebras is presented by
homogeneous relations of degree $N$ rather than two \cite{berger}).
We recall that a graded algebra over a field $k$ is AS-Gorenstein if
\[
\Ext_A^i(k,k)\cong \begin{cases}
                    0 & i\ne n\\
		    k &	i=n
                   \end{cases}
\]
It is clear by extending the proof of proposition 4.3 in
\cite{taillefer} that graded twisted Calabi-Yau algebras over a field
are AS-Gorenstein.  In the other direction, the result of
Dubois-Violette implies that, in the Koszul case, AS-Gorenstein
algebras are twisted Calabi-Yau.  If we work over a general semisimple
algebra $S$ instead of $k$, then the same relation holds between
twisted Calabi-Yau and the AS-Gorenstein property with $k$ replaced by $S$.
In this paper, we will therefore use the twisted Calabi-Yau condition.

One of the main goals of this paper is to generalize \cite{dbv} to the
several-vertex case.  Precisely, we give a Morita-invariant
construction of algebras $\cD(\omega, k)$ over any semisimple
$\C$-algebra (\S \ref{cofree}).  Using this, we show that algebras
which occur in the higher-dimensional McKay correspondence also derive
from a superpotential (\S \ref{mckay}).  We give a method to compute
the superpotential for the path algebra with relations which is Morita
equivalent to $\C[V]\# G$ and illustrate this with some examples.
These results generalize those of Crawley-Boevey and Holland
\cite{CBHolland} \cite{dmvlect} and Ginzburg \cite{GB} in the cases
$G=GL_2, GL_3$.  We then prove (\S \ref{koszul}) that any $N$-Koszul,
(twisted) Calabi-Yau algebra over a semisimple algebra is of the form
$\cD(\omega, k)$, where $\omega$ is a (twisted) superpotential.  This
last theorem generalizes \cite[Theorem 11]{dbv} to the quiver case,
and gives another proof of the fact that McKay correspondence algebras
are given by a (twisted) superpotential.  More generally, we show
that $N$-Koszul twisted Calabi-Yau algebras are equivalent to algebras
$A = \cD(\omega, k)$ such that an associated complex \eqref{ucplx} yields
a bimodule resolution of $A$.

We end by illustrating this theorem in the case of Sklyanin algebras
of dimension 4 (\S \ref{sklyanin}), which was the main motivating
example behind Section \ref{koszul}. We give a formula for the
superpotential (which was done in \cite[\S 6.4]{dbv}, see also
\cite{dbvfirst}, in different language and over $\R$). In \S
\ref{modskly}, we describe the twisted superpotentials associated to
the algebras from \cite{Staff} related to the Sklyanin algebras. We
explain that the Sklyanin McKay-correspondence algebras involve
subgroups of the Heisenberg group over $\Z/4$.

As an application of our results, we give a simple
representation-theoretic computation of the moduli space of Sklyanin
algebras of dimension 4 (Theorem \ref{sklymod}).
This description involves considering the projective space of
superpotentials. Since the automorphism group of a generic Sklyanin is
a form of the Heisenberg group over $\Z/4$ equipped with the
Heisenberg representation (which is uniquely determined by the action
of its center), we are able to find a version of the Weil
representation acting on superpotentials.  Pulling this back, we
obtain a description of the moduli space in terms of the original
parameters for the Sklyanin algebras.

We remark that, while it is probably possible to obtain this result
using the geometry associated to Sklyanin algebras (an elliptic curve
and a point of that curve), and our result can also be deduced from,
e.g., \cite[Proposition 3.1, \S 9]{con2}, it is interesting that the
theorem follows purely from representation-theoretic consequences of
the action of the Heisenberg group by automorphisms on the Sklyanin
algebra.

\subsection{Acknowledgements}
We thank M. Dubois-Violette for kindly pointing out to us his paper
\cite{dbv} and other references, for suggesting to consider
$N$-Koszulity, and answering many questions.  We also thank R. Hadani
for useful discussions about the Weil representation.  The second and
third authors would like to thank the University of Antwerp for
hospitality while part of this work was done.

\section{Coordinate-free potentials}\label{cofree}
In this section we formulate potentials, derivations, and
$\cD(\omega, k)$ in a categorical way for a tensor algebra
over a semisimple algebra.
\subsection{Duals, Duals, Duals...}\label{duals}
Let $S$ be a finite-dimensional semisimple algebra over $\C$ and let $V$ be an $S$-bimodule.
There are at least 4 distinct way to construct a dual bimodule to $V$:
\begin{itemize}
 \item The space of linear morphisms to $\C$: $V^* := \Hom_{\C}(V,\C)$ with bimodule action is $(s_1\psi s_2)(w)= \psi(s_2 w s_1)$.
 \item The space of right-module morphisms to $S$: $V^{*R} := \Hom_{\Mod S}(V,S)$ with  bimodule action is
$(s_1\psi s_2)(w)= s_1\psi(s_2 w )$.
 \item The space of left-module morphisms to $S$: $V^{*L} := \Hom_{S-\Mod}(V,S)$ with bimodule action is
$(s_1\psi s_2)(w)= \psi(w s_1)s_2$.
 \item The space of bimodule morphisms to $S \otimes_{\C} S$: $V^{*B} := \Hom_{\Bimod S}(V,S \otimes_\C S)$. Using Sweedler notation, we write
$\psi \in V^{*B}$ as $\psi_1 \otimes \psi_2$, with bimodule action
$(s_1\psi s_2)_1(w)\otimes (s_1\psi s_2)_2(w)=\psi_1(w)s_2\otimes s_1\psi_2(w)$.
\end{itemize}
These duals extend all to $4$ contravariant functors $*,*R,*L,*B:\Bimod S \to \Bimod S$.
All these different constructions are not canonically isomorphic in the category of $S$-bimodules, so
in order to identify them we need an extra datum. This extra datum is a nondegenerate trace function
on $S$. This function $\Tr: S \to \C$ allows us define natural isomorphisms $L,R,B$ from the complex dual to the
the $3$ other duals by demanding that for $\psi \in V^*$
\[
\forall w \in V:  \psi(w) = \Tr R\psi (w) = \Tr L\psi (w) = \Tr((B\psi)_1(w))\Tr((B\psi)_2(w)).
\]
Moreover, these identifications are compatible with Morita
equivalence: if $e \in S$ is an idempotent such that $SeS = S$, then
the trace on $S$ restricts to a nondegenerate trace on $eSe$.  The
images of the identification maps under the Morita equivalence
$\cM: \Bimod S \to \Bimod eSe $ are precisely the identification maps
of the restricted trace.  

From now on we will fix a trace on $S$ and omit the functors.  For
$\psi \in V^*$ and $w \in V$, we will denote the canonical pairing by
\[
 [\psi w] = [w \psi] = \psi(w).
\]

This yields $S$-bimodule morphisms $[]: V^* \otimes_S V \to S$ and $[]: V \otimes_S V^* \to S$ called the evaluation maps. The duals of these maps are called the coevaluation maps:
\[
 \coev_R: S \to V\otimes_S V^* \text{ and } \coev_L: S \to V^*\otimes_S V
\]
Ve will write the image of $1$ under the coevaluation as formally as
\[
\coev_R(1) = \sum_{Rx} x \otimes_S x^* \text{ and }\coev_L(1) = \sum_{Lx} x^* \otimes_S x
\]
These elements satisfy the following standard evaluation-coevaluation identities:
\se{
\forall \zeta \in V^*: \zeta &= \sum_{Rx} [\zeta x]x^* = \sum_{Lx} x^*[x\zeta]\\
\forall u \in V: u &= \sum_{Rx} x[x^*u] = \sum_{Lx} [ux^*]x.
}

The bracket notation can be extended to tensor products of $V$ and $V^*$ to obtain maps $[]:V^{*\otimes k} \times V^{\otimes l} \to V^{\otimes l-k}$ (for $l \geq k$)
such that
\[
[\phi_1\otimes \dots \otimes \phi_k~w_1\otimes \dots \otimes w_l]=[\phi_1[\phi_2 \dots [\phi_k w_1] \dots w_{k-1}]w_k]\cdot w_{k+1}\otimes \dots \otimes w_{l},
\]
and similarly $[]: V^{\otimes l} \times V^{*\otimes k}  \to V^{\otimes l-k}$.
If $k=l$ we end up with an element in $S$ and we can take the trace to obtain a pairing $\<,\>$ between $V^{*\otimes k}$ and $V^{\otimes k}$.
For $k>l$, we may replace the image $V^{\otimes l-k}$ by $V^{*\otimes k-l}$. 
These satisfy associativity identities, e.g., 
$[(\phi \otimes \psi)x] = [\phi [\psi  x]]$ and $[[\phi x]\psi]=[\phi[x\psi]]$
if $\psi \in V^{*\otimes k}$, $\phi \in V^{*\otimes l}$ and $x \in V^{\otimes n}$ with $n\ge k+l$.

\subsection{Potentials}
A \emph{weak potential} of degree $n$ is an element of degree $n$ in the tensor algebra $T_SV$ that commutes with the $S$-action:
\[
\omega \in  V^{\otimes n} \text{ such that }\forall s \in S: s\omega=\omega s.
\]

A weak potential is called a \emph{superpotential} if
\[
 \forall \psi \in V^*: [\psi \omega] = (-1)^{n-1}[\omega\psi].
\]

Let $\tau$ be a graded $\C$-algebra automorphism of $T_SV$ that keeps the trace invariant. This gives us an automorphism of $S$ as a $\C$-algebra, and we can define for any bimodule $M$ the left twist ${}_\tau M$ to be the vector space $M$ equipped with the bimodule action $s_1\cdot x \cdot s_2 := s_1^\tau x s_2$. The right twist $M_\tau$ is defined analogously.  We obtain isomorphisms ${}_{\tau^{-1}} S \cong S_\tau, {}_{\tau^{-1}} V \cong V_{\tau}$ using $\tau$, and
${}_{\tau^{-1}} V^* \cong (V^*)_{\tau}$ using $\tau^*$.

We then define a \emph{twisted weak potential} of degree $n$ to be an element
\[
\omega \in V^{\otimes n}  \text{ such that }\forall s \in S: s^\tau\omega=\omega s.
\]

A twisted superpotential is an element $\omega$ satisfying
\[
 \forall \psi \in V^*: [\psi^{\tau^*} \omega] = (-1)^{n-1}[\omega\psi].
\]

For every (twisted) weak potential $\omega$ and every $k$, we can define a bimodule morphism
\[
\Delta_k^\omega: (V^{\otimes k})^* \otimes S_\tau \to V^{\otimes n-k}: \psi \otimes x \to [\psi \omega x].
\]
We will denote the image of $\Delta_k^\omega$ by $W_{n-k} \subset V^{\otimes n-k}$.

\begin{definition} \label{cddfn}
  We define the derivation-quotient algebra of $\omega$ of order $k$ to be the
path algebra modulo the ideal generated by the $S$-bimodule $W_{n-k}$:
\[
\cD(\omega,k) := \C Q/\<\Image \Delta^\omega_k\> =\C Q/\<W_{n-k}\>.
\]
Here, $\<M\>$ stands for the smallest two-sided ideal containing $M$.
\end{definition}

\subsection{Path Algebras and Quivers}\label{quivers}
Let us look at all these concepts in case of a path algebra of a quiver.
A quiver $Q$ consists of a set of vertices $Q_0$ a set of arrows $Q_1$ and two maps $h,t: Q_1\to Q_0$ assigning to every
arrow its head and tail. We define $S=\C^{Q_0}$ where the vertices form a basis of idempotents, we equip it with a trace $\Tr$ such that all vertices have trace $1$.
We construct the $S$-bimodule  $V = \C^{Q_1}$ such that for every arrow $a$ we have the identity $a=h(a)at(a)$.
The path algebra can now be seen as $\C Q := T_SV$. Note that with this notation, the composition of the arrows is given by
\[
ab = \stackrel{a}{\l}\stackrel{b}{\l}.
\]

The basis $\{a\}$ gives us a dual basis $\{a^*\}$ for $V^*$, and these bases can be tensored to get dual basis for the space of (co)paths of length $k$: $\C Q_k = V^{\otimes k}$ and $V^{*\otimes k}$. The brackets have the following form:
\se{
\<a^*,b\> = \delta_{ab}, ~[a^*b] = \delta_{ab}t(b), \text{ and }[ba^*] = \delta_{ab}h(b).
}
More generally, if $p,q$ are paths, then we obtain that bracketing corresponds to taking partial derivatives:
\[
 \partial_p q = [p^*q].
\]
A weak potential is an element in $\C Q_k$ that consists only of closed paths (i.e. $h(p)=t(p)$) and $\Delta_k^\omega$ corresponds to the map $(\C Q_k)^* \to \C Q_{d-k}: p^* \to \partial_{p}\omega$.
It is a superpotential if $[a^*\omega]=(-1)^{n-1}[\omega a^*]$ which is the same as saying that $\vec\omega = (-1)^{n-1}\omega$, where  $\vec\omega$ denotes the cyclic shift: $\vec{a_1\dots a_n}=a_na_1\dots a_{n-1}$.

If $\tau$ is an automorphism of $\C Q$ then a twisted weak potential consist of a linear combination of paths $p$ that
satisfy $h(p)=\tau(t(p))$. It is a twisted superpotential if $[a^{\tau *}\omega]=(-1)^{n-1}[\omega a^{*}]$ which is the same as saying that $\vec\omega^\tau = (-1)^{n-1}\omega$, where $\vec\omega^\tau$ is the twisted cyclic shift: $\vec{a_1\dots a_n}^\tau=a_n^{\tau}a_1\dots a_{n-1}$.

\subsection{Morita Equivalence}\label{morita}
The new formulation has the advantage that it is compatible with standard Morita equivalence:
\begin{lemma}
Let $e \in S$ be an idempotent such that $SeS=S$. If $M \subset T_SV$ is an $S$-bimodule then there is a Morita equivalence between $A = T_S V/\<M\>$ and
\[
T_{eSe} eVe/\<eMe\>
\]
and if $\omega$ is a (twisted) weak potential and $e^\tau=e$ then we have that
\[
e\cD(\omega,k)e = \cD(e\omega e,k)
\]
\end{lemma}
\begin{proof}
By standard Morita equivalence between $S$ and $eSe$, we have a functor
\[
\cF:\Bimod S \to \Bimod eSe
\]
 which maps $M$ to $eMe$.
This functor commutes with tensor products
$\cF(M \otimes_S N) \cong  \cF(M) \otimes_{eSe} \cF(N)$ where $e(m\otimes_S n)e \mapsto eme\otimes_{eSe} ene$ is the natural isomorphism.
The same holds for duals and direct sums. This
implies that $\cF(T_S V)=eT_SVe \cong T_{eSe}eVe$ and if we have an $S$-sub-bimodule $M \subset  T_S V$ we also have that $\cF(M) \subset \cF(T_SV)$
and $\cF(\<M\>) = \<\cF(M)\>$. This gives us an isomorphism between
$T_{eSe} eVe/\<eMe\>$ and $eT_SV/\<M\>e$ which is Morita equivalent to $T_S V/\<M\>$.

Note that if we have a left $S$-module morphism between two bimodules $f:V_1 \to V_2$ we can consider this as an element in the bimodule $V_1^* \otimes_S V_2$. The map $\cF(f)$ can be identified with $efe \in \cF(V_1^* \otimes_S V_2)= \cF(V_1)^*\otimes \cF(V_2)$.
In the case of superpotentials we get $M = \Image \Delta_k^\omega $ and $\cF(M)= \cF(\Image \Delta_k^\omega) = \Image e \Delta_k^\omega e$
but
\[
\Delta_k^\omega : \phi \otimes x \mapsto [\phi \omega x] \text{ so }e\Delta_k^\omega e : e\phi e \otimes exe \mapsto [e\phi  e\omega e  x e]
\]
and hence $e \Delta_k^\omega e= \Delta_k^{e\omega e}$.
\end{proof}

\section{McKay correspondence and potentials}\label{mckay}

Let $G$ be any finite group, and let $V$ be an arbitrary finite
dimensional representation. We can look at the tensor algebra
$T_\C V^*$ and the ring of polynomial functions on $V$: $\C[V]$.
This last ring can be seen as the $n-2^{nd}$-derived algebra
coming from the superpotential:
\[
\omega = \sum_{\sigma \in \Symm_n} (-1)^\sigma
x_{\sigma(1)}\otimes \dots  \otimes x_{\sigma(n)} \in T_\C V^*.
\]
where $x_1\dots x_n$ form a basis for $V^*$. Indeed for every path
$p=x_{i_1}\dots x_{i_{n-2}}$ we get that $\partial_p \omega$ is
zero if some of the $x_{i_{\cdots}}$ are identical and otherwise
it is equal to the commutator between the two basis elements that
do not occur in $p$. We conclude
\[
 \C[V] \cong \cD(\omega, n-2).
\]

If $R$ is a ring with $G$ acting as automorphisms we can construct
the smash product $R \# G$. As a vector space this ring is
isomorphic to $R \otimes_\C \C G$ and the product is given by
\[
(r_1 \otimes g_1)\cdot (r_2 \otimes g_2) = r_1(g_1\cdot r_2)
\otimes g_1g_2.
\]

For the tensor algebra $TV^*$ we can rewrite the smash product as
a tensor algebra over the group algebra $\C G$. Let us define $U =
V^* \otimes_\C\C G$. The $\C G$-bimodule action on it is given by
\[
 g(v \otimes x)h := gv \otimes gxh.
\]
It is easy to see that for every $k$ we have \se{
 (T_\C V^* \# G)_k &\cong  V^* \otimes_\C \dots \otimes_\C V^* \otimes_\C \C G \\
&\cong (V^* \otimes_{\C} \C G) \otimes_{\C G} \dots \otimes_{\C G}
(V^* \otimes_{\C} \C G)= (T_{\C G} U)_k. } The special bimodule
action on $U$ makes the identifications also compatible with the
product, so that $T_\C V^* \# G \cong T_{\C G} U$. So the smash of
the tensor algebra is again a tensor algebra but now over the
semisimple algebra $\C G$. This algebra is isomorphic to
\[
 \bigoplus_{S_i} \Mat_{\dim S_i \times \dim S_i}(\C),
\]
where we sum over all simple representations of $G$. The standard
traces of this matrix algebras provide us a trace on $\C G$.

\begin{lemma}
If $R \cong T_\C V^*/\<M\>$ where $M$ is a vector space of
relations which is invariant under the $G$-action on $T_\C V^*$
then
\[
R \# G \cong T_{\C G} U / \< M \otimes_\C \C G\>.
\]
\end{lemma}
\begin{proof}
If $M$ is a $G$-invariant vector space in $T_\C V^*$ then
$M\otimes_\C \C G$ can be considered as a $\C G$-subbimodule of
$T_\C V^* \# G$. This means that if $\ideal i \lhd T_\C V^*$ is a
$G$-invariant ideal then $\ideal i \otimes_\C \C G$ is an ideal of
$T_\C V^* \# G$. Moreover if $\ideal i =\<M\>$ with $M$ a
$G$-invariant subspace of $T_\C V^*$ then $\ideal i \otimes \C G
=\<M \otimes_\C \C G\>$. So
\[
\frac{T_{\C G} (V^*\otimes_\C \C G)}{\<M \otimes_\C \C G\>} =
\frac{(T_{\C} V^*) \otimes_\C \C G}{\<M\> \otimes_\C \C G}=
\frac{T_\C V^*}{\<M\>}\otimes_\C \C G = R \# G. \qedhere
\]
\end{proof}
Suppose $R=\C[V]$ with its action of $G$. Now $\C \omega \cong \wedge^n
V^*$ is a one-dimensional $G$-representation. This means that
$\wedge^n V^* \otimes_\C \C G$ is a bimodule of the form $\C
G^\tau$ where $\tau(g) = (\det g) g$ and hence the element $\omega
\otimes_\C 1$ is a twisted weak potential. It is easy to check
that
\[
(\Image\Delta_k^\omega) \otimes \C G = \Image ((\Delta_k^\omega)
\otimes \Id_{\C G}) = \Image(\Delta_k^{(\omega \otimes 1)}).
\]
Furthermore we see that the $\tau$ changes the blocks in $\C G$
coming from simple representations $S_i$ and $\wedge^n V^* \otimes
S_i$, therefore it is easy to find an $e=\sum e_i$ such that
$e^\tau=\tau$. Also $\Tr(e_i)=1$ just as we want it to be for a
path algebra.

Finally the twisted weak potential is a superpotential because the
original $\omega$ is, and the property of being a superpotential is
preserved by taking $\omega$ to $\omega \otimes_\C 1$.

We deduce the following main result:
\begin{theorem}\label{maintheorem}
The algebra $\C[V]\#G$ is a derivation-quotient algebra of order $n-2$ with
a (twisted if $G \not \subset \SL_n$) superpotential of degree $n$.
The same is true for the corresponding Morita equivalent path
algebra with relations.
\end{theorem}

How do we work out the potential in terms of paths in the path
algebra? If $G$ is a finite group acting on $V$ then the quiver
underlying $e(\C[V]\#G) e$ is called the McKay quiver. Its vertices
$e_i$ are in one to one correspondence to the simple
representations $S_i$ of $G$. We can consider the $e_i$ as minimal
idempotents sitting in $\C G$ such that $e=\sum e_i$ and $S_i
\cong \C G e_i$. The trace function on $\C G$ then allows us to
identify $\C G^*$ with $\C G$ as $\C G$-bimodules: $\C G \to \C
G^*: g \mapsto \Tr(g\cdot -)$. Therefore $S_i^*$ is isomorphic to
$e_i\C G$ as a right module.

The number of arrows from $e_i$ to $e_j$ is equal to the dimension
of
\[
e_j V^*\otimes \C G e_i = \Hom_{\C G}(\C G e_j, (V^*\otimes \C
G)e_i)= \Hom_{\C G}(S_j, (V^*\otimes S_i)).
\]
This means that we can (and do) identify each arrow
$a:e_i\rightarrow e_j$ with a certain intertwiner morphism
$\psi_a: S_{h(a)} \to V^* \otimes S_{t(a)}$.

The set of arrows gives a basis of these intertwiner maps and we
have a dual basis $a^*$,  which can be interpreted as a collection of maps
\[
 \psi_{a^*}: S_{t(a)} \to V \otimes S_{h(a)},
\]
using the natural pairing between $\Hom_{\C G}(S_j, (V^*\otimes
S_i))$ and  $\Hom_{\C G}(S_i, (V\otimes S_j))$.

If we have a nontrivial twist, $\tau$, we make sure that the basis
we choose for the arrows is closed under the twist. We can do this
by tensoring the $\psi$-maps with the representation
$\wedge^nV^{(*)}$: \se{ \psi_{a^\tau}=\id_{\wedge^nV}
\otimes_\C\psi_{a^\tau}: \id_{\wedge^nV} \otimes_\C S_{h(a)} \to
V^* \otimes_\C \id_{\wedge^nV} \otimes S_{t(a)}. }

The composition of two arrows $ab$ gives rise to a composition
\[
\xymatrix@C=40pt{S_{h(ab)}\ar[r]^(0.4){\psi_b}&V^*\otimes
S_{t(b)=h(a)}\ar[r]^{\Id_{V^*}\otimes\psi_a}&V^*\otimes V^*\otimes
S_{t(ab)}}
\]
and
\[
\xymatrix@C=40pt{S_{t(ab)}\ar[r]^(0.4){\psi_{a^*}}&V\otimes
S_{h(a)=t(b)}\ar[r]^{\Id_{V}\otimes\psi_{b^*}}&V\otimes V\otimes
S_{h(ab)}}.
\]
In this way we can assign to every path $p$ of length $k$ a map
$\psi_p : S_{h(p)} \to V^{*\otimes k} \otimes S_{t(p)}$ and a map
$\psi_{p^*} : S_{t(p)} \to V^{\otimes k} \otimes S_{h(p)}$.

For every $k\leq n$ we have an antisymmetrizer: $\alpha^k :=
V^{\otimes k} \to \wedge^k V : v_1 \otimes \dots \otimes v_k
\mapsto v_1 \wedge \dots \wedge v_k$. If $p$ is a path of length
$n$ consider the map
\[
\xymatrix@C=40pt{ S_{t(p)}\ar[r]^(0.4){\psi_{p^*}}& V^{\otimes k}
\otimes S_{h(p)}\ar[r]^{\alpha^n \otimes \Id_{S_{h(p)}}}& \wedge^n
V \otimes S_{h(p)}}.
\]
The first factor in the image is a one-dimensional
$G$-representation which we can denote by $\det_V$. Taking the
tensor product with $\det_V$ induces a permutation on the simple
representations and hence on the vertices of the McKay quiver:
\[
 e_i = \tau(e_j) \iff S_{e_i}\cong S_{e_j} \otimes {\det}_V.
\]
By Schur's lemma $(\alpha^n \otimes \Id_{S_{h(p)}})\psi_{p^*}$ is
zero if $\tau(h(p))\ne t(p)$ and else it is a scalar; in both
cases we denote it by $c_p \in \C$.

These scalars allow us to write down an explicit form of the
superpotential. The weak potential $\omega\otimes 1$ in $T_{\C G}
(V^* \otimes \C G)$ acts as a linear function on $( \C G^* \otimes
V )^{\otimes_{\C G} n}=V^{\otimes n}\otimes \C G$: $(\omega\otimes
1)(v\otimes x)=\Tr(\omega(v)x)$. But, because $\omega \in \wedge^n
V^*$, we see that $\omega(v) = \alpha^n (v)$. Hence, if we apply
$\psi_{p^*}$, we get $(\omega\otimes
1)(\psi_{p^*})=c_p\,\, \t{dim\ }h(p)$. Because the Morita equivalence
between $\C G$ and $e\C G e$ is compatible with taking the dual, we
see that
\[
 e(\omega \otimes 1)e = \sum_{|p|=n }e(\omega \otimes 1)e(\psi_{p^*})p = \sum_{|p|=n }(c_p\,\, \t{dim\ }h(p))
 p:=\Phi
\]
and so $\C [V]\# G$ is Morita equivalent to
\[
\frac{T_{eSe} eUe}{\<\Image(\Delta_{n-2}e(\omega \otimes 1)e)\>}
\cong \frac{\C Q}{\<
\partial_q\Phi: q \t{ is path of length } n-2\>}.
\]

\section{Corollaries and Remarks}\label{corols}
In this section we show how the main result of the last section
recovers several known results in the literature. In particular we
show that for a finite subgroup of $\SL(2,\C)$, we recover the
preprojective algebra; for a finite small subgroup of $GL(2,\C)$ we
recover the mesh relations; and for a finite subgroup of $\SL(3,\C)$ we
recover the superpotential in Ginzburg \cite{GB}.  Furthermore if the
group is abelian in $GL(n,\C)$ we can also recover the toric result.

Recall our convention that when referring to quivers, $xy$ means
$y$ followed by $x$.

We start with the toric case: suppose $G$ is a finite abelian
subgroup of $GL(n,\C)$.  Being abelian we may choose a basis
$e_1,\hdots,e_n$ of $V$ that diagonalizes the action of $G$ and
thus we get $n$ characters $\rho_1,\hdots,\rho_n$ defined by
setting $\rho_i(g)$ to be the $i^{th}$ diagonal element of $g$. It
is clear that $e_i$ is a basis for the representation $\rho_i$.

In what follows it is convenient to suppress tensor product signs
as much as possible, so we write $\rho_{i,j}\rho$ for
$\rho_i\otimes\rho_j\otimes\rho$.   In this notation
$det_V=\rho_{1,\hdots,n}$.   Denote the set of irreducible
representations by $Irr(G)$.

\begin{corollary}[\cite{CMT1}]
Let $G$ be a finite abelian subgroup of $GL(n,\C)$. Then the McKay
quiver is the directed graph with a vertex for each irreducible
representation $\rho$ and an arrow
\[
\xymatrix{\rho_{i}\rho\ar[r]^{x_i^{\rho}} & \rho}
\]
for all $1\leq i\leq n$ and $\rho\in Irr(G)$.   Furthermore, the
path algebra of the McKay quiver modulo the relations
\[
\{ x_j^{\rho}x_i^{\rho_j\rho}=x_i^{\rho}x_j^{\rho_i\rho}: \rho\in
Irr(G), 1\leq i,j\leq n \}
\]
is isomorphic to the skew group ring $\C [V]\#G$.
\end{corollary}
\begin{proof}
The first statement regarding the McKay quiver is trivial since
$V=\rho_1\oplus\hdots\oplus\rho_n$.  Furthermore since $G$ is
abelian the idempotent $e$ in \S 3 is the identity and so we
really are describing the skew group ring up to isomorphism, not
just Morita equivalence.

For the relations, we build a potential as follows:  first recall
we have a basis $e_1,\hdots,e_n$ of $V$ (from which $e_i$ is a
basis for each $\rho_i$).  Since the $\rho_i$ generate the group
of characters this gives a basis for every representation. Now if
we view the map $x_i^{\rho}$ as an intertwiner
$\rho_{i}\otimes\rho\rightarrow V\otimes\rho$ it is clear that it
can be represented as simply $e_i\otimes v_\rho\mapsto e_i\otimes
v_\rho$ where $v_\rho$ is the basis element of $\rho$.

This means that if a path $p:det_V\otimes\rho\rightarrow \rho$ of
length $n$ contains two $x$'s with the same subscript then
$c_p=0$. Consequently, for any given $\rho\in Irr(G)$, the only
non-zero contributions to the potential from paths
$det_V\otimes\rho \rightarrow\rho$ of length $n$ come from
\[
\xymatrix{det_V\otimes\rho\ar[r]^(0.65){x_{\sigma(1)}^{\rho_{\sigma(2),\hdots,\sigma(n)}\rho}}&\rho_{\sigma(2),\hdots,\sigma(n)}\rho
\ar[r]^(0.65){x_{\sigma(2)}^{\rho_{\sigma(3),\hdots,\sigma(n)}\rho}}&\rho_{\sigma(3),\hdots,\sigma(n)}\rho\ar[r]
& \hdots\ar[r]&
\rho_{\sigma(n)}\rho\ar[r]^(0.65){x_{\sigma(n)}^{\rho}}& \rho }
\]
where $\sigma \in \Symm_n$.  Thus for each $\rho\in Irr(G)$ we
obtain a contribution to the potential
\[ \Phi_\rho:=\sum_{\sigma \in
\Symm_n}\alpha^n(e_{\sigma(1)}\otimes\hdots\otimes
e_{\sigma(n)})x_{\sigma(n)}^{\rho}x_{\sigma(n-1)}^{\rho_{\sigma(n)}\rho}\hdots
x_{\sigma(2)}^{\rho_{\sigma(3),\hdots,\sigma(n)}\rho}
x_{\sigma(1)}^{\rho_{\sigma(2),\hdots,\sigma(n)}\rho}.
\]
Adding these contributions one obtains the potential
$\Phi=\sum_{\rho\in Irr(G)}\Phi_\rho$.  It is easy to see that
differentiating $\Phi$ with respect to paths of length $n-2$ give
the required relations.
\end{proof}
As another corollary to Theorem \ref{maintheorem} we have
\begin{corollary}[Reiten-Van den Bergh \cite{Reiten-VdB}]\label{mesh} Suppose $G$ is a finite
subgroup of $GL(2,\C)$ without pseudoreflections. Then the
relations on the McKay quiver which give a Morita equivalence with
$\C[x,y]\# G$ are precisely the mesh relations from AR theory on
$\C[[x,y]]^G$ and the superpotential is exactly the sum of all
mesh relations.

In particular for a finite subgroup of $\SL(2,\C)$, the
preprojective algebra of the corresponding extended Dynkin diagram
is Morita equivalent to $\C[x,y]\# G$.
\end{corollary}
\begin{proof}
We will work out the proof in the completed case and then go back
by taking the associated graded ring. Denote by $R=\C[[x,y]]$ the
ring of formal power series in two variables and consider the
Koszul complex over R
\[
\xymatrix{0\ar[r]&R\otimes det_V\ar[r]& R\otimes V\ar[r]& R\ar[r]&
\C\ar[r]&0}.
\]
We know this comes from a superpotential.  We proved that the
algebra obtained by smashing with a group $G$ also comes from a
(possibly twisted) superpotential, so
\[
\xymatrix{0\ar[r]&R\otimes det_V\otimes \C G\ar[r]& R\otimes
V\otimes \C G\ar[r]& R\otimes \C G\ar[r]& \C G\ar[r]&0}
\]
(which is the minimal projective resolution of the $R\# G$ module
$\C G$) arises from a superpotential, i.e. the relations on $R\#
G$ can be read off from the fact that the composition of the first
two maps is zero.

For convenience label the members of $Irr(G)$ by
$\sigma_0,\sigma_1,\hdots,\sigma_n$ where $\sigma_0$ corresponds
to the trivial representation.  Since $\C
G=\oplus_{i=0}^{n}\sigma_i^{\oplus dim(\sigma_i)}$ the above exact
sequence decomposes into
\[
\oplus_{i=0}^{n}(\xymatrix@C=15pt{0\ar[r]&R\otimes det_V\otimes
\sigma_i\ar[r]& R\otimes V\otimes \sigma_i\ar[r]& R\otimes
\sigma_i\ar[r]& \sigma_i\ar[r]&0})^{\oplus dim(\sigma_i)}
\]
so really the relations on $R\# G$ can be read off from the fact
that the composition of the first two maps in each summand is
zero.  But now \cite{Yoshino1}[10.9]
\[
\begin{array}{rcl}
\mathfrak{proj} R\# G & \approx & \mathfrak{CM} R^G\\
M&\mapsto& M^G
\end{array}
\]
is an equivalence of categories, where $\mathfrak{proj} R\# G$ is
the category of finitely generated projective $R\# G$ modules, and
$\mathfrak{CM} R^G$ is the category of maximal Cohen-Macaulay
modules for $R^G$. Thus taking $G$-invariants of the above exact
sequence, the relations on $R\# G$ can be read off from the fact
that the composition of the first two maps in each summand of
\[
\oplus_{i=0}^{n}(\xymatrix@C=10pt{0\ar[r]&(R\otimes det_V\otimes
\sigma_i)^G\ar[r]& (R\otimes V\otimes \sigma_i)^G\ar[r]& (R\otimes
\sigma_i)^G\ar[r]& \sigma_i^G\ar[r]&0})^{\oplus dim(\sigma_i)}
\]
is zero. It is clear that $\sigma_i^G=0$ for $i\neq 0$ whilst
$\sigma_0^G=\C$.  But now by \cite{Yoshino1}[10.13] for $i\neq 0$
the summands above are precisely the AR short exact sequences, and
for $i=0$ the sequence has the appropriate AR property. Thus the
relations on $eR\# Ge$ are precisely the mesh relations.

Because the mesh relations are graded and taking the associated
graded is compatible with the Morita equivalence we can conclude
that the relations of $e\C[x,y]\# G e$ are also given by the mesh
relations and the superpotential will be the sum of all mesh
relations.
\end{proof}

Because we work with superpotentials there is a redundancy in the
coefficients of the potential:
\begin{lemma}\label{cylic_invariant}
Chose a basis for the arrows in $\C Q=e(T_\C V \# G)e$ that is
closed under the application of the twist $\tau$. Then the
coefficients of  $e(\omega \otimes 1)e=\sum_{|p|=n} c_p p$ have
the following property: if $p=p_1\hdots p_n$ is a path of length
$n$ then
\[
c_{p_1\hdots p_n}=(-1)^{n-1} c_{p_n^\tau p_1p_2\hdots p_{n-1}}
\]
\end{lemma}
\begin{proof}
This follows immediately from Theorem \ref{maintheorem} and the
discussion in subsection \ref{quivers}.
\end{proof}

Note that if $G\leq \SL(n,\C)$, the twist is trivial so we can work
with any basis for the arrows. In this case, not only does the
above lemma simplify the calculation of the $c_p$, but it also
tells us that we can write our superpotential up to cyclic
permutation. This generalizes a result of Ginzburg \cite{GB} for
$\SL(3,\C)$.

Note that care has to be taken when translating between our fully
written superpotentials and the more compact versions in terms of
cyclic notation.  For example if $u$ is a non-trivial path of
length 1 which forms a cycle at some vertex, then if in our
potential we have $c_{u\hdots u}u\hdots u$ (where there are $n \in
2\N+1$ $u$'s) then in cyclic notation this should be written as
$\frac{c_{u\hdots u}}{n}(u\hdots u)$ since the cycle counts the
element $n$ times.

The superpotential highly depends on the representatives we chose
for the arrows in $Q$. From the point of view of the quiver we
have an action of
\[
Aut_{\C Q_0} \C Q = \prod_{i,j \in Q_0} GL(i(\C Q)_1j).
\]
on the space $(\C Q)_n$ and all potentials that give an isomorphic
derivation-quotient algebra are in the same orbit. An interesting open
question is whether there one can always find a nice
representative for the superpotential.

\section{Examples of McKay Correspondence Superpotentials}\label{examples}
In this section we illustrate Theorem \ref{maintheorem} by computing
examples.  We first illustrate that our theorem does not depend on
whether or not $G$ has pseudoreflections by computing an example
of a non-abelian group $G\leq GL(2,\C)$ where $\C^2/G$ is smooth:

\begin{example}\label{D8} Consider the dihedral group  $D_8=\langle g,h: g^4=h^2=1,
h^{-1}gh=g^{-1} \rangle$ viewed inside $GL(2,\C)$ as
\[
g=\begin{pmatrix}\e_4&0\\0&\e_4^{-1}\end{pmatrix},
h=\begin{pmatrix}0&1\\1&0\end{pmatrix}
\]
It is clear that the invariant ring is $\C[xy,x^4+y^4]$ and so is
smooth.  Denoting the natural representation by $V$, the character
table for this group is
\[
\begin{array}{c|rrrrr}
&  1 & g^2 & g & h & gh\\ \hline
 V_{0} & 1& 1&1 &1 &1\\
 V_{1}& 1& 1& 1&-1 &-1\\
 V_{2}& 1& 1&-1&1 &-1\\
 V_{3}& 1& 1&-1&-1&1\\
 V  & 2&-2&0 &0 &0\\
\end{array}
\]
and so the McKay quiver has the shape
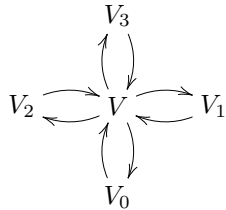
\[
\xymatrix@C=20pt@R=20pt{ &V_{3}\ar@/^0.5pc/[d]&&\\
V_{2}\ar@/^0.5pc/[r]&
V\ar@/^0.5pc/[d]\ar@/^0.5pc/[u]\ar@/^0.5pc/[l]\ar@/^0.5pc/[r]&
 V_{1}\ar@/^0.5pc/[l] \\
&V_{0}\ar@/^0.5pc/[u]&&}
\]
We shall show that the
algebra Morita equivalent to the skew group ring is:
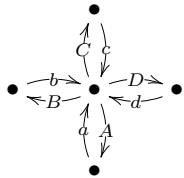
\[
\begin{array}{cc}\begin{array}{c}
\xymatrix@C=20pt@R=20pt{ &\bullet\ar@/^0.35pc/[d]|{c}&&\\
\bullet\ar@/^0.35pc/[r]|{b}&
\bullet\ar@/^0.35pc/[d]|{A}\ar@/^0.35pc/[u]|{C}\ar@/^0.35pc/[l]|{B}\ar@/^0.35pc/[r]|{D}&
 \bullet\ar@/^0.35pc/[l]|{d} \\
&\bullet\ar@/^0.35pc/[u]|{a}&&}\end{array} &
\begin{array}{c}\begin{array}{cc}Da=0 & Cb=0 \\Ad=0&Bc=0
\end{array}\\
aA+dD=bB+cC
\end{array}
\end{array}
\]
Note that $\tau(a)=d, \tau(d)=a, \tau(A)=D, \tau(D)=A$ and
likewise with the $c$'s and $D$'s.   Notice also that there are 5
relations, which coincides with the number of paths of length 0
(i.e. the number of vertices).  We now check that the relations
guessed above are correct:

Take the following $G$-equivariant basis:
\begin{alignat*}{2}
V_0 \otimes V&=\C(v_0\otimes e_1) + \C(v_0\otimes e_2)    &\qquad &(V\sim A)\\
V_1 \otimes V&=\C(v_1\otimes -e_1) + \C(v_1\otimes e_2) &\qquad &(V\sim D)\\
V_2 \otimes V&=\C(v_2\otimes e_2) + \C(v_2\otimes e_1)   &\qquad &(V\sim B)\\
V_3 \otimes V&=\C(v_3\otimes -e_2) + \C(v_3\otimes e_1)   &\qquad &(V\sim C)\\
V\otimes V&=\C\left(e_1 \otimes e_2 + e_2\otimes e_1 \right)&\qquad &(V_0\sim a)\\
& \quad +\C\left( -e_1 \otimes e_2 +  e_2 \otimes e_1 \right)& \qquad &(V_1\sim d)\\
& \quad +\C\left( e_1 \otimes e_1 +  e_2 \otimes  e_2 \right)&\qquad &(V_2\sim b)\\
& \quad +\C\left( -e_1 \otimes e_1 +  e_2 \otimes e_2
\right)&\qquad &(V_3\sim c)
\end{alignat*}
Since the determinant representation is $V_1$, if we consider
paths of length 2 ending at a given vertex $\rho$, the only
possible ones with non-zero $c_p$ must start at $\rho\otimes V_1$.
Consequently our search for non-zero $c_p$ restricts to the
following cases:
\[
\begin{array}{cc|c}
\t{start vertex} & \t{isomorphism} & \t{end vertex}\\ \hline
V_1\cong V_0\otimes V_1 & v_1\mapsto v_0\otimes v_1 &V_0\\
V_0\cong V_1\otimes V_1 & v_0\mapsto v_1\otimes v_1 &V_1\\
V_3\cong V_2\otimes V_1 & v_3\mapsto v_2\otimes v_1 &V_2\\
V_2\cong V_3\otimes V_1 & v_2\mapsto v_3\otimes v_1 &V_3\\
V\cong V\otimes V_1 & \begin{array}{rcl} e_1&\mapsto&e_1\otimes-v_1\\
e_2&\mapsto&e_2\otimes v_1\end{array} &V
\end{array}
\]
With this information the intertwiners are easy to compute: for
example
\[
\xymatrix{ V_1\ar[r]^(0.4)d&V\otimes V\ar[r]^(0.4){A\otimes 1}&
V_0\otimes V\otimes V\ar[r]^(0.6){1\otimes\alpha^2} & V_0\otimes
V_1\ar[r]^{\cong}& V_1}
\]
takes
\[
v_1\mapsto -e_1\otimes e_2+e_2\otimes e_1\mapsto -v_0\otimes
e_1\otimes e_2+v_0\otimes e_2\otimes e_1\mapsto -2v_0\otimes
v_1\mapsto -2v_1
\]
and so $c_{Ad}=-2$.  Continuing in this fashion our potential
(after dividing through by 2) is
\[
-Da+aA-Ad+dD+Cb-bB+Bc-cC
\]
which in compact form may be written as $-(Da)^\tau+(Cb)^\tau$.
Since $n-2=0$ we don't differentiate and so these are precisely
the relations, thus we obtain the relations guessed above.
\end{example}
\begin{remark} Taking a different $G$-equivariant basis may lead to a
potential which is not invariant under twisted cyclic permutation.
\end{remark}
\begin{remark} In the above example if we change $h$ slightly and so our group is
now the binary dihedral group $\mathbb{D}_{3,2}$ generated by
\[
a=\begin{pmatrix}\e_4&0\\0&\e_4^{-1}\end{pmatrix},
b=\begin{pmatrix}0&1\\-1&0\end{pmatrix}
\]
inside $\SL(2,\C)$, then although the character table and so shape
of the McKay quiver is the same, the relations differ.  Indeed, by
Lemma~\ref{mesh} the relations are now the preprojective
relations. This can also be verified directly by choosing an
appropriate $G$-equivariant basis.
\end{remark}

We now illustrate Lemma \ref{mesh} with an example of a finite
small subgroup of $GL(2,\C)$:
\begin{example}
Take $G=\mathbb{D}_{5,2}$, i.e. the group inside $GL(2,\C{})$
generated by
\[
G=\langle \begin{pmatrix} \e_4 & 0 \\ 0 & \e_4^{-1}
\end{pmatrix},\begin{pmatrix} 0 & \e_4 \\ \e_4 & 0
\end{pmatrix},\begin{pmatrix} \e_6 & 0 \\ 0 & \e_6
\end{pmatrix} \rangle
\]
The McKay quiver is
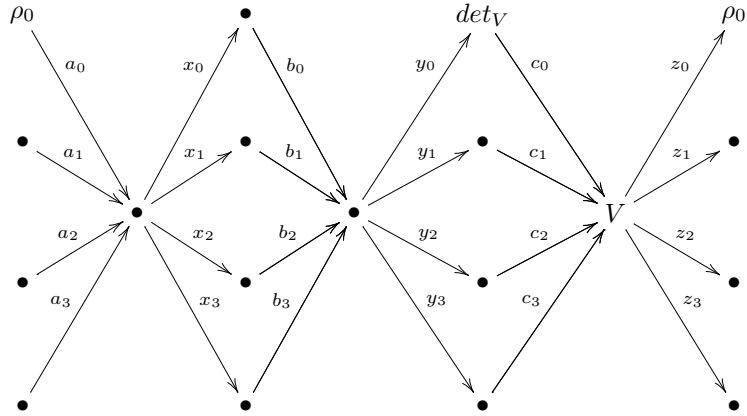
\[
\xymatrix@C=30pt@R=15pt{\rho_0\ar[3,1]^(0.3){a_0}&&\bullet\ar[3,1]\ar[3,1]^(0.3){b_0}&&det_V\ar[3,1]\ar[3,1]^(0.3){c_0}&&\rho_0\\
&&&&&&\\
\bullet\ar[1,1]^(0.35){a_1}&&\bullet\ar[1,1]\ar[1,1]^(0.35){b_1}&&\bullet\ar[1,1]\ar[1,1]^(0.35){c_1}&&\bullet\\
&\bullet\ar[-3,1]^(0.7){x_0}\ar[-1,1]^(0.65){x_1}\ar[1,1]^{x_2}\ar[3,1]^{x_3}&&\bullet\ar[-3,1]^(0.7){y_0}\ar[-1,1]^(0.65){y_1}\ar[1,1]^{y_2}\ar[3,1]^{y_3}&&V\ar[-3,1]^(0.7){z_0}\ar[-1,1]^(0.65){z_1}\ar[1,1]^{z_2}\ar[3,1]^{z_3}&&\\
\bullet\ar[-1,1]^{a_2}&&\bullet\ar[-1,1]\ar[-1,1]^{b_2}&&\bullet\ar[-1,1]\ar[-1,1]^{c_2}&&\bullet\\
&&&&&&\\
\bullet\ar[-3,1]^{a_3}&&\bullet\ar[-3,1]\ar[-3,1]^{b_3}&&\bullet\ar[-3,1]\ar[-3,1]^{c_3}&&\bullet}
\]
where the trivial, determinant and natural representations are
illustrated, and the ends of the two sides are identified. Note
that the permutation $\tau$ induced by tensoring with the
determinant representation rotates this picture to the left, and
so the fact that the permutation coincides with the AR translate
is implicit. The mesh relations are
\[
\begin{array}{ccccc}
x_0a_0=0 && y_0b_0=0 && z_0c_0=0\\
x_1a_1=0 && y_1b_1=0 && z_1c_1=0\\
x_2a_2=0 && y_2b_2=0 && z_2c_2=0\\
x_3a_3=0 && y_3b_3=0 && z_3c_3=0
\end{array}
\qquad
\begin{array}{rcl}
b_0x_0+b_1x_1+b_2x_2+b_3x_3&=&0\\
c_0y_0+c_1y_1+c_2y_2+c_3y_3&=&0\\
a_0z_0+a_1z_1+a_2z_2+a_3z_3&=&0
\end{array}
\]
and so we have 15 relations, matching the number of paths of
length 0 (i.e. the number of vertices).
\end{example}

\begin{example}
Take $G=\frac{1}{7}(1,2,4)\ltimes \langle \tau \rangle$, i.e. the
group inside $\SL(3,\C{})$ generated by
\[
G=\langle \begin{pmatrix} \e & 0 & 0\\ 0 & \e^2 &0\\0&0&\e^4
\end{pmatrix},\begin{pmatrix} 0& 1 & 0\\ 0 & 0 &1\\1&0&0
\end{pmatrix} \rangle
\]
where $\e^7=1$.  The McKay quiver is
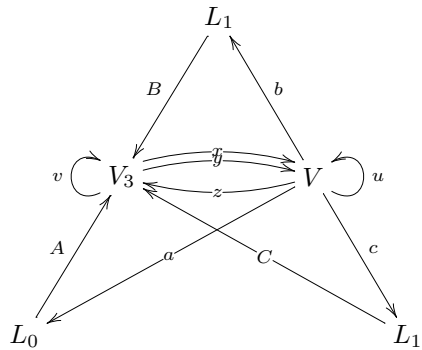
\[
\xymatrix@C=20pt@R=20pt{&&L_1\ar[2,-1]_B&&\\
&&&&\\
&V_3\ar@(ld,lu)^{v}\ar@/^0.5pc/@<0.75ex>[0,2]|x\ar@/^0.5pc/[0,2]|y & & V\ar@/^0.5pc/[0,-2]|z\ar[-2,-1]_b\ar[2,1]^c\ar[2,-3]|a\ar@(dr,ur)_{u} &\\
&&&&\\
L_0\ar[-2,1]^A&&&&L_1\ar[-2,-3]|C}
\]
Denote the basis of $L_i$ by $l_i$ for $1\leq i\leq 3$, the basis
$V$ by $e_1,e_2,e_3$ and the basis of $V_3$ by $j_1,j_2,j_3$.
Taking the following $G$-equivariant basis:
\begin{alignat*}{2}
L_0 \otimes V&=\C(l_0\otimes e_1) + \C(l_0\otimes e_2) + \C(l_0\otimes e_3) &\qquad &(V\sim a)\\
L_1 \otimes V&=\C(l_1\otimes \rho e_1) + \C(l_1\otimes \rho^2 e_2) + \C(l_1\otimes e_3) &\qquad &(V\sim b)\\
L_2 \otimes V&=\C(l_2\otimes \rho^2 e_1) + \C(l_2\otimes \rho e_2) + \C(l_2\otimes e_3) &\qquad &(V\sim c)\\
V\otimes V&=\C(e_3 \otimes e_3) + \C(e_1\otimes e_1)+\C(e_2\otimes e_2) &\qquad &(V\sim u)\\
& \quad +\C(e_1 \otimes e_2) + \C(e_2 \otimes e_3) +\C(e_3\otimes e_1)& \qquad &(V_3\sim x)\\
& \quad +\C(e_2 \otimes e_1) + \C(e_3 \otimes e_2) +\C(e_1\otimes e_3)&\qquad &(V_3\sim y)\\
V_3\otimes V&=\C\left(j_1 \otimes e_3 + j_2\otimes e_1 + j_3\otimes e_2 \right)&\qquad &(L_0\sim A)\\
& \quad +\C\left(j_1 \otimes \rho^2 e_3 + j_2\otimes \rho e_1 + j_3\otimes e_2 \right)& \qquad &(L_1\sim B)\\
& \quad +\C\left(j_1 \otimes \rho e_3 + j_2\otimes \rho^2 e_1 + j_3\otimes e_2 \right)&\qquad &(L_2\sim C)\\
& \quad +\C(j_2 \otimes e_2) + \C(j_3 \otimes e_3) +\C(j_1\otimes e_1)& \qquad &(V\sim z)\\
& \quad +\C(j_2 \otimes e_3) + \C(j_3 \otimes e_1) +\C(j_1\otimes
e_2)&\qquad &(V_3\sim v)
\end{alignat*}
a calculation shows that the superpotential can be written as
\[
\cycle\quad a(x-y)A+b(x-\rho y)B+c(x-\rho^2
y)C-zux+vzy+\frac{1}{3}uuu-\frac{1}{3}vvv
\]
where $\rho$ is a cube root of unity.  Differentiating with
respect to the paths of length $3-2=1$ gives the relations
\[
\begin{array}{cc}
\partial_A & ax=ay\\
\partial_B & bx=\rho by\\
\partial_C & cx=\rho^2 cy\\
\partial_a & xA=yA\\
\partial_b & xB=\rho yB\\
\partial_c & xC=\rho^2 yC\\
\partial_x & Aa+Bb+Cc=zu\\
\partial_y & Aa+\rho Bb+\rho^2 Cb=vz\\
\partial_u & xz=u^2\\
\partial_v & zy=v^2
\end{array}
\]
\end{example}

\begin{example}
As in Example \ref{D8} consider the group $D_8$, but now acting on
the representation $V\oplus V$.  Since $D_8$ is generated inside
$V$ by pseudoreflections it follows that inside $V\oplus V$ it is
generated by symplectic reflections, thus in this case $\C[V \oplus V]\# G$
is the undeformed symplectic reflection algebra.  The McKay quiver
is now
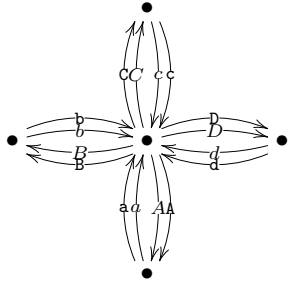
\[
\xymatrix@C=40pt@R=40pt{ &\bullet\ar@/^0.35pc/[d]|{c}\ar@<0.5ex>@/^0.55pc/[d]|{\tt{c}}&&\\
\bullet\ar@/^0.35pc/[r]|{b}\ar@<0.5ex>@/^0.55pc/[r]|{\tt{b}}&
\bullet\ar@/^0.35pc/[d]|{A}\ar@<0.5ex>@/^0.55pc/[d]|{\tt{A}}\ar@/^0.35pc/[u]|{C}\ar@<0.5ex>@/^0.55pc/[u]|{\tt{C}}
\ar@/^0.35pc/[l]|{B}\ar@<0.5ex>@/^0.55pc/[l]|{\tt{B}}\ar@/^0.35pc/[r]|{D}\ar@<0.5ex>@/^0.55pc/[r]|{\tt{D}}&
 \bullet\ar@/^0.35pc/[l]|{d}\ar@<0.5ex>@/^0.55pc/[l]|{\tt{d}} \\
&\bullet\ar@/^0.35pc/[u]|{a}\ar@<0.5ex>@/^0.55pc/[u]|{\tt{a}}&&}
\]
The superpotential is given in compact form by {\tiny{
\[
\begin{array}{rrrrrrrr}
({\tt{A}}a{\tt{A}}a)& -(A{\tt{a}\tt{A}}a)& -2({\tt{A}\tt{d}}Da)
&({\tt{A}}d{\tt{D}}a) &(A{\tt{d}}{\tt{D}}a)& ({\tt{A}}b{\tt{B}}a)&
-(A{\tt{b}}{\tt{B}}a)&-({\tt{A}}c{\tt{C}}a)\\
(A{\tt{c}}{\tt{C}}a)& (A{\tt{a}}A{\tt{a}})& ({\tt{A}}dD{\tt{a}})&
(A{\tt{d}}D{\tt{a}})& -2(Ad{\tt{D}\tt{a}})& -({\tt{A}}bB{\tt{a}})&
(A{\tt{b}}B{\tt{a}})&
({\tt{A}}cC{\tt{a}})\\
-(A{\tt{c}}C{\tt{a}})& ({\tt{D}}d{\tt{D}}d)&
-(D{\tt{d}}{\tt{D}}d)& -({\tt{D}}b{\tt{B}}d)&
(D{\tt{b}}{\tt{B}}d)& ({\tt{D}}c{\tt{C}}d)& -(D{\tt{c}}{\tt{C}}d)&
(D{\tt{d}}D{\tt{d}})\\
({\tt{D}}bB{\tt{d}})& -(D{\tt{b}}B{\tt{d}})&
-({\tt{D}}cC{\tt{d}})& (D{\tt{c}}C{\tt{d}})& ({\tt{B}}b{\tt{B}}b)&
-(B{\tt{b}}{\tt{B}}b)& -2({\tt{B}}{\tt{c}}Cb)&
({\tt{B}}c{\tt{C}}b)\\
(B{\tt{c}}{\tt{C}}b)& (B{\tt{b}}B{\tt{b}})& ({\tt{B}}cC{\tt{b}})&
(B{\tt{c}}C{\tt{b}})& -2(Bc{\tt{C}}{\tt{b}})&
 {\tt{C}}c{\tt{C}}c)& -(C{\tt{c}}{\tt{C}}c)& (C{\tt{c}}C{\tt{c}})
\end{array}
\]}}where recall since we are inside $\SL(4,\C)$ a negative sign is
introduced with cyclic permutation.  Differentiating appropriately
gives the relations {\tiny{\[
\begin{array}{c}
\begin{array}{cccc}
\begin{array}{c}
Da=0\\ {\tt{Da}}=0\\
{\tt{D}}a=-D{\tt{a}}\\
{\tt{D}}b=D{\tt{b}}\\ {\tt{D}}c=D{\tt{c}}\\ {\tt{D}}d=D{\tt{d}}
\end{array} &
\begin{array}{c}
Ad=0\\ {\tt{Ad}}=0\\
{\tt{A}}a=A{\tt{a}}\\ {\tt{A}}b=A{\tt{b}}\\
{\tt{A}}c=A{\tt{c}}\\ {\tt{A}}d=-A{\tt{d}}
\end{array}
 &
 \begin{array}{c}
Cb=0\\ {\tt{Cb}}=0\\
{\tt{C}}a=C{\tt{a}}\\ {\tt{C}}b=-C{\tt{b}}\\
{\tt{C}}c=C{\tt{c}}\\ {\tt{C}}d=C{\tt{d}}
\end{array} &
\begin{array}{c}
Bc=0\\ {\tt{Bc}}=0\\
{\tt{B}}a=B{\tt{a}}\\
{\tt{B}}b=B{\tt{b}}\\ {\tt{B}}c=-B{\tt{c}}\\ {\tt{B}}d=B{\tt{d}}
\end{array}
\end{array} \\
aA+bB=cC+dD \\{\tt{aA+bB}}={\tt{cC+dD}} \\
{\tt{a}}A+b{\tt{B}}=a{\tt{A}}+{\tt{b}}B={\tt{c}}C+d{\tt{D}}=c{\tt{C}}+{\tt{d}}D=\Sigma\\
\t{where } \Sigma=
\frac{1}{2}(a{\tt{A}}+b{\tt{B}}+c{\tt{C}}+d{\tt{D}})=\frac{1}{2}({\tt{a}}A+{\tt{b}}B+{\tt{c}}C+{\tt{d}}D)
\end{array}
\]}}
The calculations involving this example were done using a computer
program written in GAP \cite{GAP4}. The source code of this
program can be downloaded at \\{\tt
http://www.algebra.ua.ac.be/research/mckay.gap}.
\end{example}
\section{Koszul algebras}
\label{koszul}
Thus far, we have explained that, for $G < GL(V)$, $\C[V] \# G$, and
hence the quiver algebras Morita equivalent to it, are twisted
Calabi-Yau and derived from a twisted superpotential (in the case
$G < \SL(V)$, we may remove the word ``twisted'').  Here we explain
that this is part of a more general phenomenon: \textit{any
  Koszul, (twisted) Calabi-Yau algebra is of the form
  $\cD(\omega, k)$}.  This was proved in \cite{dbv} for algebras over
a field, so our result generalizes this to the quiver case.  We also
prove a converse: any algebra of the form $A=\cD(\omega, k)$ is Koszul
and (twisted) Calabi-Yau iff a natural complex attached to $\omega$ is
a bimodule resolution of $A$.  

Recall that a graded algebra is Koszul if all the maps in its bimodule
resolution have degree $1$. This is clearly invariant under a Morita
equivalence $A \sim eAe$, using the functor described in Section
\ref{cofree}. Then, McKay correspondence algebras are Koszul, by the
following well-known lemma:
\begin{lemma}\label{mckaykoszul}
If $G \subset \SL(V)\cong \SL_n$ then $\C[V]\# G$ is $n$-CY and Koszul.
\end{lemma}
\begin{proof}
  The standard Koszul bimodule resolution for $\C[V]$ gives a self-dual
  resolution of $\C[V]$, so $\C[V]$ is $n$-CY. The $k^{th}$ term of
  this resolution $\cK^\bullet$ is
  $\C[V] \otimes_\C \wedge^k V^* \otimes_\C \C[V]$ and it is
  isomorphic to the $(n-k)^{th}$ dual term because of the pairing
\[
\wedge^k V^* \times \wedge^{n-k} V^*\to \C:(v_1,v_2) \to a \iff \phi_1 \wedge \phi_2  = a x_1\wedge\dots \wedge x_n.
\]
Because $G \subset \SL(V)$ this pairing is a pairing of
left $\C G$-modules. 

Now we smash the whole resolution over $\C$ with $\C G$. This tensor
functor is exact so we get a new resolution. This is now self-dual as
$\C[V] \# G$-bimodules over the base ring $\C[G]$ (i.e., as
$(\C[V] \# G) \otimes_{\C[G]} (\C[V] \# G)$-modules).  The Koszul
property follows from the fact that smashing preserves the grading.
\end{proof}

In order to formulate our main theorem, we need to introduce a natural
complex $\cW^{\bullet}$ attached to any superpotential $\omega$, for
$A = \cD(\omega,k)$ (which may be viewed as a bimodule version of
\cite[(5.3)]{dbv}).  For simplicity, we will assume for now that
$|\omega| = k+2$, so that $A$ is quadratic.

Recall the spaces $W_i$ defined just above Definition \ref{cddfn}.
Consider the complex
\begin{equation} \label{ucplx}
\cW^\bullet :=  0 \rightarrow A \otimes W_{|\omega|} \otimes A \mathop{\rightarrow}^{d_{|\omega|}} A \otimes W_{|\omega|-1} \otimes A
  \rightarrow \cdots \rightarrow  A \otimes W_1 \otimes A \mathop{\rightarrow}^{d_1} A \otimes W_0 \otimes A \rightarrow 0,
\end{equation}
where, for $v_1, \ldots, v_i \in W$ and $a, a' \in A$,
\begin{gather*}
  d_i = \varepsilon_i (\spl_{L} + (-1)^{i} \spl_{R})|_{A \otimes W_i \otimes A}, 
\\  \spl_{L}(a \otimes v_1 v_2 \cdots v_i \otimes a') = a v_1 \otimes
  v_2 \cdots v_i \otimes a', 
\\ \spl_{R}(a \otimes v_1 v_2 \cdots
  v_i \otimes a') = a \otimes v_1 \cdots v_{i-1} \otimes v_i a', \\
\varepsilon_i := \begin{cases} (-1)^{i (|\omega|-i)}, & \text{if $i < (|\omega|+1)/2$,} \\
     1, & \text{otherwise}.
   \end{cases}
\end{gather*}
It is easy to check that the above yields a complex, i.e.,
$d_i \circ d_{i+1} = 0$. Moreover, the terms, aside from $A$ itself,
are projective bimodules, and the maps are $A$-bimodule maps.  We will
see that it is exact iff $A$ is Koszul and Calabi-Yau.  More precisely, we will prove:
\begin{theorem}\label{kosthm} An algebra $T_S W / \langle R \rangle$ 
is Koszul and Calabi-Yau iff it is of the form $\cD(\omega, k)$ and
the corresponding complex \eqref{ucplx} is exact in positive degree and $H^0(\cW^\bullet)=A$.  In this case, \eqref{ucplx} is the Koszul resolution of $A$, and is self-dual.
\end{theorem}
\begin{remark}
The condition that $\cW^\bullet$ is a resolution of $A$ is very subtle and hard to check for a given potential.
In the one vertex case it is shown in \cite{dbvfirst,dbv} that this implies some special regularity conditions on $\omega$.
\end{remark}

We begin with the
\begin{lemma}\label{selfdual}
For any superpotential $\omega$, the complex \eqref{ucplx}
is self-dual.
\end{lemma}

\begin{proof}
First, note that $\omega$ induces perfect pairings
\begin{equation*}
\<\,,\>: W_{|\omega|-i}^* \otimes W_i^* \rightarrow \C, \quad \<\xi, \eta\> := [(\xi \otimes \eta) \omega],
\end{equation*}
satisfying the supersymmetry property,
\begin{equation*}
\<\xi, \eta\> = (-1)^{|\eta| |\xi|} \< \eta, \xi\>.
\end{equation*}
This yields an isomorphism $\eta: W_{|\omega|-i}^* \iso W_{i}$, and hence
a duality pairing of bimodules
\begin{equation*}
\<\,,\>: (A \otimes W_i \otimes A) \otimes (A \otimes W_{|\omega|-i} \otimes A) \rightarrow A \otimes A, \quad
\<a \otimes x \otimes a', b \otimes y \otimes b'\> :=  a'b \otimes [\eta^{-1}(x) y] \otimes b'a.
\end{equation*}
This explains why the terms in the above complex are in
duality.

It remains to check that the differentials satisfy the self-duality
property: $d_i = d_{|\omega|+1-i}^*$.  It suffices to show that
$\spl_L|_{A \otimes W_i \otimes A}$ is identified with
$\varepsilon_i \varepsilon_{|\omega|-i} \spl_R^*|_{A \otimes
  W_{|\omega|-i}^* \otimes A}$
under the above duality.  That is, if we denote by
$\<\,,\>^{-1}: W_i \otimes W_{|\omega|-i} \rightarrow
\C$ the inverse to the pairing $\<\,,\>$, then for all
$x \in W_{i}, y \in W_{|\omega|+1-i}$, we need to check that
\begin{equation*}
\< 1 \otimes y \otimes 1, \spl_{L}(1 \otimes x \otimes 1)\>  = \varepsilon_i \varepsilon_{|\omega|-i} \< \spl_R(1 \otimes y \otimes 1), 1 \otimes x \otimes 1 \>.
\end{equation*}
This amounts to checking that, for all $\xi \in W_1^* = W^*$, we have
\begin{equation*}
\< [\xi x], y \>^{-1} = \varepsilon_i \varepsilon_{|\omega-i|} \< x, [y \xi] \>^{-1},
\end{equation*}
where $\<\,,\>^{-1}$ denotes the inverse pairing to $\<\,,\>$, i.e.,
\begin{equation*}
\<u, w\>^{-1} := \<\eta^{-1}(u), \eta^{-1}(w)\> = \varepsilon_i \varepsilon_{|\omega|-i} [u \eta^{-1}(w)] = (-1)^{|u| |w|} [w \eta^{-1}(u)].
\end{equation*}
Thus, we have to check that
\begin{equation*}
[\eta^{-1}([\xi x]) y] = \varepsilon_i \varepsilon_{|\omega|-i} [x \eta^{-1}[y \xi]] = \varepsilon_i \varepsilon_{|\omega|-i} (-1)^{|x| (|y|-1)} [\eta^{-1}(x) [y \xi]].
\end{equation*}
By associativity identities and the definition of $\eta$, the left-hand
side is $[(\eta^{-1}(x) \otimes \xi) y]$, which is equal to
$\varepsilon_i \varepsilon_{|\omega|-i} (-1)^{i(|\omega|-i)}$ times the RHS.  Thus, setting 
$\varepsilon_i = \begin{cases} (-1)^{i (|\omega|-i)}, & \text{if $i < (|\omega|+1)/2$,} \\
     1, & \text{otherwise}
   \end{cases}$ yields the desired self-duality.
\end{proof}

\begin{lemma}\label{subcomplex}
The complex \eqref{ucplx} is a subcomplex of the Koszul complex for $\cD(\omega, |\omega|-2)$.
\end{lemma}
\begin{proof}
  The Koszul complex can be defined as follows. If $A =T_S W/\<R\>$
  where $R$ is an $S$-subbimodule of $W\otimes W$, then we denote by
  $R^\perp$ the submodule of $W^*\otimes W^*$ that annihilates $R$.
  The Koszul dual of $A$ is $A^! := T_S W^*/\<R^\perp\>$ and it is
  again a graded algebra. For each $k$ we have a projection
  $W^{*\otimes k} \to A^!_k$, and, dually, this gives us injections
  $(A^!_k)^* \to W^{\otimes k}$.  The Koszul complex $\cK^\bullet$ is
  defined by the maps
  $d:A\otimes (A^!_k)^* \otimes A\to A\otimes (A^!_{k-1})^* \otimes A$
  which are constructed analogously to the maps in \eqref{ucplx}. To
  prove the lemma we only have to show that $W_k \subset (A^!_k)^*$.

What does $(A^!_k)^*$ look like? Because $A^!_k = W^{*\otimes k}/(\sum_l W^{*\otimes l} \otimes R^\perp \otimes W^{*\otimes k-l-2})$ one has that $w \in (A^!_k)^*$
if and only if $\<w,\phi\>=0$ for all $\phi \in (\sum_l W^{*\otimes l} \otimes R^\perp \otimes W^{*\otimes k-l-2})$. This is the same as to say that
\[
 w \in \bigcap_l W^{\otimes l} \otimes R \otimes W^{\otimes k-l-2} = \bigcap_l W^{\otimes l} \otimes W_2 \otimes W^{\otimes k-l-2}.
\]
We conclude immediately that $W_k \subset (A^!_k)^*$.
\end{proof}
\begin{remark}
According to theorem 5.4 in \cite{marconnet}, the pairing between the $W_k$ can be extended to a pairing between
the $(A^!_k)^*$, but this pairing is not necessarily perfect. Only in the AS-Gorenstein (Koszul twisted Calabi-Yau) case this pairing becomes perfect because then $(A^!_k)^*$ and $W_k$ coincide. 
\end{remark}

\begin{proof}[Proof of Theorem \ref{kosthm}]
  If \eqref{ucplx} is a resolution of $A$ then
  since all of the differentials have degree $+1$ with respect to the
  grading of $A$, this would imply (by one definition of Koszulity)
  that $A$ is Koszul, and that \eqref{ucplx} is a Koszul resolution of
  $A$ (more generally, for any graded algebra, any free bimodule
  resolution of $A$ with differentials of positive degree must be
  minimal and unique).  Then, by Lemma \ref{selfdual}, $A$ is
  Calabi-Yau as well.

  Conversely, suppose that $A$ is CY($n$) and Koszul. Using the
  CY($n$) property, \cite[Theorem A.5.2]{bocklandt} shows that there
  is a trace function $\Tr: \Ext^n_A(S, S) \rightarrow \C$ such that
 \begin{equation}\label{trssym}
\Tr(\alpha * \beta) = (-1)^{k(n-k)} \Tr(\beta * \alpha), \alpha \in \Ext^k(S,S), \beta \in \Ext^{n-k}(S,S)
\end{equation} 
induces a perfect pairing, where $*$
denotes the Yoneda cup product.  Using the Koszul property, 
we may identify $\Ext^n(S, S)$ with a quotient of $(W^*)^{\otimes n}$,
so that a trace function becomes canonically an element $\omega \in W^{\otimes n}$.  Then, \eqref{trssym} says precisely that $\omega$ is a superpotential.

By nondegeneracy, the trace pairing induces an isomorphism
$\Ext^2(S, S) \cong \Ext^{n-2}(S, S)^*$.  Furthermore,
$\Ext^2(S, S) \cong R$, so this isomorphism translates into the
statement that $W_2 = R$.  Thus, $A \cong \cD(\omega, n-2)$.

Moreover, for the same reason, $\Ext^i(S, S) \cong W_i$ for all $i$, and
hence \eqref{ucplx} must be exact.  Thus it is the minimal=Koszul resolution
of $A$.
\end{proof}
\begin{remark}
  This theorem is a generalization of Theorem \ref{maintheorem} (at
  least in the nontwisted case): to obtain Theorem \ref{maintheorem},
  we combine Theorem \ref{kosthm} and Lemma \ref{mckaykoszul}.
\end{remark}

Finally, we explain briefly how to generalize to $N$-Koszul and
twisted Calabi-Yau algebras.  First, for the twisted Calabi-Yau and
twisted superpotential setting, all that changes is that the twisted
superpotential property proves a twisted self-duality in Lemma
\ref{selfdual}, and conversely in the proof of Theorem \ref{kosthm}.

Next, for the $N$-Koszul setting, first recall \cite{berger} that an
$N$-Koszul algebra is an algebra $A$ presented by homogeneous relations of degree $N$ so that there is a free resolution of $A$ with differentials of degrees
alternating between $N-1$ and $1$:
\begin{equation*}
\cdots \rightarrow A \otimes Y_2 \otimes A \mathop{\rightarrow}^{d_2} A \otimes Y_1 \otimes A \mathop{\rightarrow}^{d_1} A \otimes A \mathop{\rightarrow}^m A \rightarrow 0,
\end{equation*}
where $d_{i}$ has degree $1$ if $i$ is odd, and $N-1$ if $i$ is even.  

In the $N$-Koszul setting, it is natural to define $N$-complexes of bimodules instead of complexes \cite{wambst}. These are sequences of maps $\cK_i\stackrel{d}{\to} \cK_{i-1}$ such that $d^N=0$ (instead of $d^2$). As in the ordinary Koszul case
one can define a selfdual bimodule $N$-complex 
\begin{equation} 
\cW^\bullet :=  0 \rightarrow A \otimes W_{|\omega|} \otimes A \mathop{\rightarrow}^{d_{|\omega|}} A \otimes W_{|\omega|-1} \otimes A
  \rightarrow \cdots \rightarrow  A \otimes W_1 \otimes A \mathop{\rightarrow}^{d_1} A \otimes W_0 \otimes A \rightarrow 0,
\end{equation}
where $d_i = \varepsilon_i (\spl_{L} + (q)^{i} \spl_{R})|_{A \otimes W_i \otimes A}$, and $q$ is the primitive $N^th$
root of $1$. 
This $N$-complex is a subcomplex of the Koszul $N$-complex as defined in \cite{marconnet}. 

The $N$-complexes can be contracted to obtain complexes. In the case of our selfdual $N$-complex we get
\begin{equation} \label{uncplx}
0 \rightarrow A \otimes W_{mN + 1} \otimes A \rightarrow A \otimes W_{mN} \otimes A \rightarrow A \otimes W_{(m-1)N+1} \otimes A \rightarrow \cdots \rightarrow A \otimes W_N \otimes A \rightarrow A \otimes W_1 \otimes A \rightarrow A \otimes W_0 \otimes A  \rightarrow 0,
\end{equation}
where the differentials alternate between
$\pm(\spl_{L}^{N-1} + \spl_{L}^{N-2} \spl_R + \cdots + \spl_{L}
\spl_R^{N-2} + \spl_R^{N-1})$
and $\pm(\spl_L - \spl_R)$.  
This is not selfdual anymore unless we are in the case that $|\omega|=mN+1$ for some $m \in N$.

The $N$-Koszul generalization of theorem \ref{kosthm} becomes
\begin{theorem}
An algebra $T_S W / \langle R \rangle$ 
is $N$-Koszul and twisted Calabi-Yau iff it is of the form $\cD(\omega, k)$ for a twisted superpotential $\omega$ (with the same twisting) and
the corresponding complex \eqref{uncplx} is exact in positive degree and $H^0(\cW^\bullet)=A$.  In this case, \eqref{uncplx} is the $N$-Koszul resolution of $A$, and is twisted self-dual.
\end{theorem}
The proof of the theorem uses basically the same arguments as in the proof of theorem \ref{kosthm} but adapted
to the $N$-Koszul situation in accordance with the result from \cite{berger} and \cite{marconnet}.

\section{Sklyanin algebras}\label{sklyanin}

In this section we show how to compute the superpotential for the
four-dimensional Sklyanin algebras as introduced by Sklyanin in
\cite{Skly1, Skly2}.  These algebras may be thought of as ``elliptic deformations'' of the polynomial algebra in four variables, and they are in particular
Koszul and have the same Hilbert series $\frac{1}{(1-t)^4}$ as the
polynomial ring.  

Following \cite{SS}, \S 0, fix values $\alpha, \beta,$ and
$\gamma$ satisfying\footnote{In the original form \cite{Skly1, Skly2},
  see also e.g.~\cite{OF89, Sm}, not all values
  $\alpha, \beta, \gamma$ satisfying this equation are
  considered---only those that arise from an elliptic curve and a
  point of that curve. By, e.g., \cite{SS}, these are the values where 
\eqref{nondegwk} holds and $\alpha, \beta, \gamma \neq 0$; cf.~Theorem
\ref{sszthm}.}
\begin{equation} \label{abgsolns}
\alpha + \beta + \gamma + \alpha \beta \gamma = 0.
\end{equation}
Then, the algebra $A$ is defined by
\begin{equation*}
A := \C \langle x_0, x_1, x_2, x_3 \rangle / I,
\end{equation*}
where $I$ is the two-sided ideal generated by the relations $r_i, s_i$,\footnote{Our notation $r_i$ is for the relation involving $x_0 x_i - x_i x_0$, and $s_i$ is the relation involving $x_0 x_i + x_i x_0$.}
\begin{gather*}
r_1 := x_0 x_1 - x_1 x_0 - \alpha (x_2 x_3 + x_3 x_2), \quad s_1 := x_0 x_1 + x_1 x_0 - (x_2 x_3 - x_3 x_2), \\
r_2 := x_0 x_2 - x_2 x_0 - \beta (x_3 x_1 + x_1 x_3), \quad s_2 := x_0 x_2 + x_2 x_0 - (x_3 x_1 - x_1 x_3), \\
r_3 := x_0 x_3 - x_3 x_0 - \gamma (x_1 x_2 + x_2 x_1), \quad s_3 := x_0 x_3 + x_3 x_0 - (x_1 x_2 - x_2 x_1).
\end{gather*}

We would like to find a superpotential for $A$.  This must be a
supercyclic element of $I$ which is homogeneous of degree four.  It is
easy to compute that, under the assumption
\begin{equation} \label{nondegwk}
(\alpha, \beta, \gamma) \notin \{(\alpha, -1, 1), (1, \beta, -1),
  (-1, 1, \gamma)\},
\end{equation}
the
space of such elements is one-dimensional and spanned by the following
element:
\begin{gather} \label{sklypot}
\omega := \kappa_1(r_1 s_1 + s_1 r_1) + \kappa_2(r_2 s_2 + s_2 r_2) + \kappa_3 (r_3 s_3 + s_3 r_3),
\intertext{where $(\kappa_1, \kappa_2, \kappa_3) \neq (0,0,0)$ is determined up to
a nonzero multiple by}
  \kappa_1(1+\alpha) = \kappa_3(1-\gamma), \quad \kappa_1(1-\alpha) = \kappa_2(1+\beta), \quad \kappa_2(1-\beta) = \kappa_3(1+\gamma).
\end{gather}
\begin{proposition} The element $\omega$ is a superpotential.
  Moreover, for any $\alpha, \beta, \gamma$ satisfying
  \eqref{nondegwk}, $A \cong \cD(\omega, 2)$, and in this case, the
  resolution \eqref{ucplx} is a self-dual resolution of $A$, making
  $A$ Calabi-Yau.
\end{proposition}
\begin{proof}
  It is easy to verify that $\omega$ is a superpotential (in fact,
  it makes sense and is cyclically supersymmetric even if \eqref{nondegwk}
  is not satisfied).  Next,
  suppose \eqref{nondegwk} holds.  Then, $\kappa_1, \kappa_2,$ and $\kappa_3$
  are nonzero.  Since the elements $r_i, s_i$ are linearly
  independent, it follows that $A \cong \cD(\omega, 2)$.

  To deduce that \eqref{ucplx} is a resolution of $A$, we make use of
  the
\begin{theorem}\cite{SS}
  Assuming \eqref{nondegwk}, $A$ is
  Koszul.  Moreover, $H(A^!, t) = (1+t)^4$.
\end{theorem}
In the above theorem, $H(V, t)$ denotes the Hilbert series of a graded
vector space $V$, i.e., $H(V,t) = \sum_{m \geq 0} \dim V(m) t^m$. The
hard part of the above theorem is the Koszulity.

Now, by Lemma \ref{subcomplex} and the formula for the Koszul complex
(see the proof of Lemma \ref{subcomplex}), it suffices only to show
that $\dim W_i = \binom{n}{i}$ for all $i$. For $i = 2$, this follows
from the above observations; then, it follows by applying partial
derivatives to the relations $r_j, s_j$ that this is true for $i=1$.
Since $i=0$ is obvious, we get $\dim W_i = \binom{n}{i}$ for all
$i$ by the self-duality of $\cW^{\bullet}$.  Thus, $A$ is Calabi-Yau
with self-dual resolution $\cW^{\bullet}$.
\end{proof}

\begin{remark}
  It is also easy to derive that $A$ is Calabi-Yau directly from
  \cite{SS}: in particular, in \cite{SS} it is shown that $A^!$ is
  Frobenius, one may easily show that $A^!$ is in fact symmetric.  Our
  contribution here is in producing a superpotential and showing that
  the minimal (Koszul) resolution of $A$ is produced in this way.
\end{remark}

\subsubsection{Modified Sklyanin algebras from \cite{Staff}} \label{modskly}
In \cite{Staff}, some new algebras related to the above are defined
and shown to be Koszul, and have the same Hilbert series
$\frac{1}{(1-t)^4}$ as the polynomial ring in four variables.  Here,
we explain that these algebras are not Calabi-Yau, but rather twisted
Calabi-Yau, with twisted superpotential described below.  We omit
the proofs, which are the same as for the Sklyanin algebra. Following
\cite{Staff}, let us assume in this subsection that
$\{\alpha, \beta, \gamma\} \cap \{0, 1, -1\} = \emptyset$.

Heuristically, these algebras are ``elliptic deformations'' of the algebra
$\C \langle x_0, x_1, x_2, x_3 \rangle / ( -x_0^2 + x_1^2 + x_2^2 +
x_3^2, x_i x_j - x_j x_i \mid \{i,j\} \neq \{2,3\} )$
in the same way that the Sklyanin algebras are deformations of
$\C[x_0, x_1, x_2, x_3]$.

Precisely, the relations are given by using any five of the relations
$r_i, s_j$, and replacing the sixth with the new relation
$q := d_1 \Omega_1 + d_2 \Omega_2$, where
\begin{equation*}
  \Omega_1 := -x_0^2 + x_1^2 + x_2^2 + x_3^2, \quad 
\Omega_2 := x_1^2 + \frac{1+\alpha}{1-\beta} x_2^2 + 
\frac{1-\alpha}{1+\gamma} x_3^2.
\end{equation*}
We obtain the algebra
$A' = \C \langle x_0, x_1, x_2, x_3 \rangle / I'$, where $I'$ is the
ideal generated by $q$ and five of the $r_i, s_j$.  (The geometric
motivation for studying $A'$ is that it and the Sklyanin algebra $A$
both surject to the same ring $B := A/(\Omega_1, \Omega_2)$ of
geometric origin.)

First, suppose that the relations are $q, r_2, r_3, s_1, s_2, s_3$ (so $r_1$ is not a relation). We claim that $A$ is twisted Calabi-Yau with twisting
$\sigma(x_0)= -x_0, \sigma(x_1)=-x_1, \sigma(x_2)=x_2,
\sigma(x_3)=x_3$,
and with unique twisted superpotential (up to scaling) given by
\begin{gather*}
\lambda_1 (q s_1 + s_1 q) + \lambda_2 (r_2 r_3 - r_3 r_2) + \lambda_3 (s_2 s_3 - s_3 s_2), 
\intertext{with $(\lambda_1: \lambda_2: \lambda_3) \in \PP^2$ determined by}
d_2 \lambda_1 = \lambda_2 (\beta \gamma + 1), \quad d_1 \lambda_1 = - \lambda_2 + \lambda_3,
\end{gather*}
provided that any nonzero solution to the above has all of
$\lambda_1, \lambda_2, \lambda_3$ nonzero, i.e., $(d_1, d_2)$ is not a
multiple of $(1,0)$ or $(1,-1-\beta \gamma)$.

Next, suppose that the relations are $q, r_1, r_2, r_3, s_2, s_3$ (so $s_1$ is not a relation).  Then, $A$ is twisted Calabi-Yau with the same twisting as above, and the unique superpotential (up to scaling) is given by
\begin{gather*}
\lambda_1 (q r_1 + r_1 q) + \lambda_2 (r_2 s_3 - s_3 r_2) + \lambda_3 (s_2 r_3 - r_3 s_2), 
\intertext{with $(\lambda_1: \lambda_2: \lambda_3) \in \PP^2$ determined by}
\alpha d_1 \lambda_1 = \lambda_2 - \lambda_3, \quad \alpha(d_1+d_2)\lambda_1 = \beta \lambda_2 + \gamma \lambda_3,
\end{gather*}
again provided all of $\lambda_1, \lambda_2, \lambda_3$ can be nonzero
(i.e., $(d_1, d_2)$ is not a multiple of $(1,\beta-1)$ or $(1,-1-\gamma)$).
Any other $A'$ can be obtained from this or the previous paragraph by a cyclic permutation of the parameters and relations.

Finally, in \cite{Staff}, also the algebra
$A'_\infty = \C\langle x_0, x_1, x_2, x_3 \rangle /
(r_2,s_2,r_3,s_3,\Omega_1, \Omega_2)$
is studied, and shown to be Koszul and have the same Hilbert series as
the polynomial ring in four variables (just as in all the other
examples).  We claim that this algebra is twisted Calabi-Yau with
twisting $\sigma(x_i) = -x_i$ for all $i$.  In other words, the
twisted superpotential $\omega$ (which is unique up to scaling) is
actually cyclically symmetric.  We omit the formula for the
twisted superpotential.


\subsection{McKay correspondence for four-dimensional Sklyanin algebras}
It makes sense to think of the potential \eqref{sklypot} as a deformed
version of the volume form in the case of the polynomial algebra in
four variables: precisely, by \S 6, the potential for a Koszul
Calabi-Yau algebra always spans the top Hochschild homology group
(here, $HH_4(A,A)$) as a free bimodule (cf.~\cite[Proposition
10]{dbv}).  As explained in \textit{op. cit.}, there is an action of
the finite Heisenberg group on $A$ by automorphisms.\footnote{Note
  that the analogue of $\SL(4)$ in this context is a quantum version
  of $\SL(4)$ associated to an elliptic $R$-matrix; see
  e.g.~\cite{Skly1, Skly2, OF89, con1,con2}.}
In fact, the full automorphism group preserving \eqref{sklypot} is
finite:
\begin{theorem}\cite[\S 2]{ssz} \label{sszthm} Assume that 
  $\alpha, \beta, \gamma \neq 0$ and \eqref{nondegwk} holds.  Then,
  the group of graded automorphisms of $A$ ($\subset \Aut(V)$) is
  isomorphic to $\widetilde H$,
\begin{equation}
1 \rightarrow \C^\times \rightarrow 
 \widetilde H \rightarrow (\Z/4 \oplus \Z/4) \rightarrow 1,
\end{equation}
except in the case $\alpha = \beta = \gamma = \pm \sqrt{-3}$, when the group
has the form $\widetilde H \ltimes \Z/3$.
\end{theorem}
By the explicit matrices given in \cite[\S 2]{ssz} (see also the end
of \S \ref{mss}), one may easily compute that the $\widetilde H$ above
has the form
$\widetilde H \cong (\C^\times \times \langle X, Y\rangle)/ ([X,Y] =
\sqrt{-1})$,
by picking lifts $X, Y$ of generators of the quotient
$\widetilde H/\C^\times \cong \Z/4 \oplus \Z/4$.

\begin{notation} We call the subgroup of $\Aut(V)$ preserving
  $\omega \in V^{\otimes 4}$ the automorphism group of $\omega$, and
  denote it by $\Aut(\omega)$.
\end{notation}
The only elements of $\C^\times$ that act trivially on $V^{\otimes 4}$
are fourth roots of unity. As a result, the automorphism group of the
superpotential $\omega$ will be finite, of size only $64$.  It turns
out this is one of the ``$\Z/4$-Heisenberg groups,'' which we describe
as follows.  Let $X, Y \in \widetilde H$ be elements as in the previous
paragraph, chosen to have the property $X^4 = Y^4 = -1$.
Then, $H$ is the group generated by $X$ and $Y$.
  It is a central extension
\begin{equation}
  1 \rightarrow \mu_4 \rightarrow H \rightarrow \Z/4 \oplus \Z/4 \rightarrow 1,
\end{equation}
where $\mu_4 \subset \C^\times$ is the subgroup of elements of order four.
A presentation for $H$ is given by
\begin{equation}
  H \cong \langle X,Y,Z \mid XZ=ZX, YZ=ZY, Z^4 = 1, X^4=Y^4=Z^2, [X,Y]=Z \rangle.
\end{equation}
We deduce the following:
\begin{proposition}
  For any $\alpha, \beta, \gamma$ as in Theorem \ref{sszthm},
  $\Aut(\omega) \cong H$, unless $\alpha=\beta=\gamma = \pm
  \sqrt{-3}$,
  in which case this group is $H \rtimes \Z/3$, where $\Z/3$ acts
  nontrivially on $H$.
\end{proposition}
As a consequence, we see that, under the assumptions of Theorem
\ref{sszthm}, $\omega \otimes 1 \in A \# H$ is still a superpotential,
and hence also gives a superpotential for any Morita equivalent
algebra to $A \# H$.  Letting $f_1, \ldots, f_m$ be a full set of
primitive idempotents (one for each irreducible representation of
$H$), and $f := f_1 + \ldots + f_m$, we then have
\begin{proposition}
The algebra $f(A \# H)f$ is Calabi-Yau.  For any subgroup $G < \widetilde H$,
$f'(A \# G)f'$ is twisted Calabi-Yau, where $f'$ is the sum of a full set
of primitive idempotents for $G$.
\end{proposition}
These algebras may be considered the elliptic McKay correspondence
algebras in dimension four, and $f(A \# H)f$ is the maximal Calabi-Yau
one, in the sense that $H$ is maximal (and so the McKay quiver is also
the largest possible).
\subsection{The case $\alpha = 0$}
The theorem \ref{sszthm} did not apply to the case that one of
$\alpha, \beta, \gamma$ is zero.  Since we only need \eqref{nondegwk}
to obtain a Calabi-Yau algebra and a potential, it is worth proving
the analogue of Theorem \ref{sszthm} in the degenerate cases
$(\alpha,\beta,\gamma) \in \{(0,\beta, -\beta), (\alpha, 0, -\alpha),
(\alpha, -\alpha, 0), (0,0,0)\}$ (and we will use this in the
next subsection). 
By symmetry, we restrict ourselves to the case $\alpha = 0$. 

It is likely that this result is known, but we did not find it in the
literature. We
remark that, in \cite[\S 1]{SS}, it is shown that these degenerate cases are
iterated Ore extensions.
\begin{theorem}\label{degauts}
(i) Assume $(\alpha, \beta, \gamma) = (0, \beta, -\beta)$ with
$\beta \neq 0$.
  Then, the graded automorphism group of $A$  is generated by $\C^\times$, the
  group $SO(2,\C)$ acting on $\Span\{x_2,x_3\}$, i.e.,
  $\Biggl\{ \begin{pmatrix} 1 & 0 & 0 & 0 \\ 0 & 1 & 0 & 0 \\ 0 & 0
    & a & b \\ 0 & 0 & -b & a \end{pmatrix} \Biggl| a^2+b^2=1 \Biggr\}$,
  and the elements $\Biggl\{\begin{pmatrix} 0 & \frac{1}{\pm \sqrt{\beta}} & 0 & 0 \\ \pm \sqrt{\beta} & 0 & 0 & 0 \\ 0 & 0
    & 1 & 0 \\ 0 & 0 & 0 & 1 \end{pmatrix}, \begin{pmatrix} i & 0 & 0 & 0 \\ 0 & -i & 0 & 0 \\ 0 & 0 & -i & 0 \\ 0 & 0 & 0 & i \end{pmatrix} \Biggr\}$, where
  $i$ denotes a square-root of $-1$.  

  (ii) If $\alpha = \beta = \gamma = 0$, then the automorphism group
  is $\C^{\times} \cdot SO(3,\C)$, with $SO(3,\C)$ the automorphism group of
  $\Span\{x_1, x_2, x_3\}$ together with its standard symmetric
  bilinear form $(x_i, x_j) = \delta_{ij}$.
\end{theorem}
\begin{proof}
  (i) The vector $r_1 = x_0 \wedge x_1$ must be preserved up to scalar
  by any automorphism, so the span of $x_0, x_1$ is preserved; then
  the only element of $\Sym^2 \Span \{x_0, x_1\}$ in the
  symmetrization of the relations is $x_0 x_1 + x_1 x_0$. Hence, any
  automorphism must send $(x_0, x_1)$ to $(\lambda x_0, \mu x_1)$ or
  $\mu x_1, \lambda x_0$.  Up to the automorphism
  $x_0 \mapsto \sqrt{\beta} x_1, x_1 \mapsto \frac{1}{\sqrt{\beta}}
  x_0, x_2 \mapsto x_1, x_3 \mapsto x_3$
  and scaling, we may assume that our automorphism $\psi$ satisfies
  $\psi(x_0) = \lambda x_0, \psi_(x_1) = \lambda^{-1} x_1$.  Since
  then $x_0 x_1 + x_1 x_0$ is preserved, looking at $s_1$, we see that
  $x_2 x_3 - x_3 x_2$ is preserved, and hence the span of $x_2, x_3$
  is preserved.

  Next, note that the relations project isomorphically to
  $\Lambda^2 V$.  Let $[x,y] := xy-yx$ denote the commutator and
  $\{x,y\} := xy+yx$ the anticommutator.  We have
  $\psi(r_2) = [\lambda x_0, \psi(x_2)] - \frac{\beta}{\lambda}
  \{\psi(x_3), x_1\}$.
  If we write $\lambda \psi(x_2) = a x_2 + b x_3$, then we must have
  $\psi(r_2) = a r_2 + b r_3$.  This implies that, restricted to
  $\Span\{x_2, x_3\}$, $\psi$ must have the form
\begin{equation}
\psi = \begin{pmatrix} \frac{a}{\lambda} & -b \lambda \\ \frac{b}{\lambda} & a \lambda \end{pmatrix}.
\end{equation}
Applying the same reasoning to $\psi(x_3)$, we deduce furthermore that $\lambda^4 = 1$. This yields the claimed description.

(ii) Let $R$ be the vector space spanned by the relations.
In the case $\alpha = \beta = \gamma = 0$, the intersection $\Lambda^2 V \cap R$ is $\Span\{x_0\} \wedge V$, and hence any automorphism $\psi$ must send $x_0$ to a multiple of itself.  Up to scaling, let us assume
that $\psi(x_0) = x_0$.  Then, the fact that the relations project isomorphically to $\Lambda^2 V$ yields a canonical isomorphism $\Lambda^2 V / (\Lambda^2 V \cap R) \iso \Span \{x_1, x_2, x_3\}$, sending $w \in \Lambda^2 V$ to the unique
element $v$ such that $w - (x_0 v + v x_0) \in R$.  This must be preserved
by any automorphism. Hence, $U := \Span\{x_1, x_2, x_3\}$ is preserved, and the
map may be written as an isomorphism $\Lambda^2 U \iso U$. Preserving this
map in particular means that two vectors which are perpendicular under
the standard form $(x_i, x_j) = \delta_{ij}$ remain perpendicular, so that 
$\psi$, restricted to $U$, must lie in $\C^\times SO(U)$.  However,
any diagonal matrix must preserve $s_1$ and hence must be the identity, so
that $\psi \in SO(U)$.  Hence the automorphism group of $A$ (now acting on
all of $V$) lies in $\C^\times SO(U)$.  On the other hand, it is clear that
this group acts by automorphisms on $A$.
\end{proof}
\begin{corollary}
  The automorphism group of the potential $\Aut(\omega)$ is generated
  by the elements listed in the theorem, except that $\C^\times$ is
  replaced by the group $\mu_4 \subset \C^\times$ of fourth roots of
  unity.
\end{corollary}
As a consequence, we may again consider $A \# G$ for any finite
subgroup $G \subset \Aut(\omega)$, which will be a Calabi-Yau algebra,
and in the case $G \subset \Aut(A)$ but not $\Aut(\omega)$, we get a
twisted Calabi-Yau algebra.  As before, one may consider the Morita
equivalent algebras and write down their potentials.

\subsection{Moduli space of four-dimensional Sklyanin algebras} \label{mss}
In this subsection we will use the theory of the Weil representation
over $\Z/4$ and the preceding results to give a simple computation of
the moduli space of Sklyanin algebras in dimension four.  Throughout,
when we say ``isomorphism'' or ``automorphism'' of Sklyanin algebras,
we mean a graded isomorphism or automorphism.

First, we note that, given any $(\alpha, \beta, \gamma)$, the algebras
associated to this triple and any cyclic permutation are isomorphic:
the permutation
$x_0 \mapsto x_0, x_1 \mapsto x_2 \mapsto x_3 \mapsto x_1$ sends the
relations for $(\alpha, \beta, \gamma)$ to the relations for
$(\gamma, \alpha, \beta)$.  Similarly, the map
$x_0 \mapsto x_0, x_1 \mapsto x_2, x_2 \mapsto -x_1, x_3 \mapsto x_3$
sends the relations for $(\alpha, \beta, \gamma)$ to the relations for
$(-\beta, -\alpha, -\gamma)$.

Hence, if we consider the $\Symm_3$ action on the surface $\cS$ given by
$\alpha + \beta +\gamma + \alpha \beta \gamma = 0$, given by multiplying
the standard permutation action by the sign representation, we get a map
\begin{equation} \label{cs3eqn}
\cS / \Symm_3 \twoheadrightarrow \{\text{Isomorphism classes of four-dimensional Sklyanin algebras}\}.
\end{equation}
\begin{theorem}\label{sklymod}
The map \eqref{cs3eqn} is a bijection.
\end{theorem}
The rest of the subsection will be devoted to the proof of the
theorem.  The main case of the theorem concerns those parameters
satisfying the conditions of Theorem \ref{sszthm}, and we will prove
the result by finding a description of the moduli space of potentials
in terms of the Heisenberg and Weil representations.
\begin{remark}
  Note that, in the locus of elements satisfying Theorem \ref{sszthm},
  the $\Symm_3$ action is free except at the two points
  $\alpha=\beta=\gamma = \pm \sqrt{-3}$.  Here, these two points form
  a two-element orbit, and the isotropy $\Z/3$ is picked up by the
  automorphism group at these points (cf.~Theorem \ref{sszthm}).
\end{remark}

First, let us handle the degenerate cases when one of
$\alpha, \beta, \gamma$ is zero.  Suppose only one is zero, and
without loss of generality, say it is $\alpha$.  Then
$(\alpha, \beta, \gamma) = (0, \beta, -\beta)$. Note that, in this
case, the automorphism group of $A$ is independent of the value of
$\beta$.  In particular, any $\psi: V \iso V$ inducing an isomorphism
$A \iso A'$ with $A'$ of the same form must normalize the connected
component of the identity of the common automorphism group, i.e.,
$\C^\times \cdot SO(2)$.  Since $\psi$ must therefore preserve the
trivial weight spaces of $SO(2)$ and either preserve or interchange
the nontrivial weight spaces, $\psi$ must have the form
$\psi = \psi' \oplus \psi''$, where
$\psi' = \psi|_{\Span\{x_0, x_1\}}$ and
$\psi'' = \psi|_{\Span\{x_2, x_3\}}$, and
$\psi'' \in O(\Span\{x_2, x_3\}$ (the orthogonal group).  Up to an
automorphism of $A$, we may assume that
$\psi'' = \begin{pmatrix} 1 & 0 \\ 0 & \varepsilon \end{pmatrix} \}$,
with $\varepsilon \in \{1,-1\}$.  By the same argument as in the proof
of Theorem \ref{degauts}, we must have that $\psi'$ is either diagonal
or strictly off-diagonal, and using the automorphism
$\begin{pmatrix} 0 & \pm \frac{1}{\sqrt{\beta}} \\ \pm \sqrt{\beta} &
  0 \end{pmatrix}$,
we may assume $\psi'$ is diagonal, say
$\psi' = \begin{pmatrix} \lambda & 0 \\ 0 & \mu \end{pmatrix}$. Using
that the relations for $A$ and $A'$ both contain $s_1, s_2, s_3$, it
follows that $\mu \lambda = \varepsilon, \mu = \lambda$, and
$\mu = \varepsilon \lambda$.  Put together, this says that
$\varepsilon = 1$ and $\mu = \lambda = \pm 1$.  This is already an
automorphism of $A$, so $\psi \in \Aut(A)$.  That is, $A$ and $A'$
already had the same relations. So \eqref{cs3eqn} is injective when
restricted to parameters $(\alpha,\beta,\gamma)$ with exactly one
parameter equal to zero.

In the case $\alpha = \beta = \gamma = 0$, it is clear that no other
triple $(\alpha, \beta, \gamma)$ yields an isomorphic algebra.  

We didn't restrict ourselves to the Calabi-Yau condition
\eqref{nondegwk}, so let us also explain the contrary cases.  First
assume $(\alpha, \beta, \gamma) = (\alpha, -1, 1)$ with
$\alpha \neq \pm 1$.  Call $R$ the span of the relations.  We quickly
compute the automorphism group of $A$ as follows.  We see that
$R$ contains the rank-two tensors
\begin{equation}
x_0 x_2 + x_1 x_3, \quad x_2 x_0 - x_3 x_1, \quad x_0 x_3 - x_1 x_2, \quad x_3 x_0 + x_2 x_1.
\end{equation}
Set $U := \Span\{x_0, x_1\}$ and $U' := \Span\{x_2, x_3\}$. Then, the
rank-two tensors form a union of an open subvariety of
$R \cap (U \otimes U')$ and an open subvariety of
$R \cap (U' \otimes U)$.  Thus, any automorphism of $A$ must either
preserve or interchange $U$ and $U'$.  Moreover, equip $U$ and $U'$
each with their standard symmetric bilinear forms.  We see that, given
nonzero vectors $w_1, w_2 \in U$, the subspace of relations
$\{w_1 w_1' - w_2 w_2' \mid w_1', w_2' \in U'\} \cap R$ is
two-dimensional iff $(w_1, w_2) = 0$.  Hence, any automorphism of $A$
which preserves $U, U'$ must also preserve their standard symmetric
bilinear forms.  Thus, $\Aut(A)$ must be a subgroup of
$(\C^\times SO(U)) \oplus (\C^\times SO(U')) \rtimes \Z/2$, where
$1 \in \Z/2$ interchanges $U$ and $U'$, e.g., it may be the element
$x_0 \mapsto x_2 \mapsto x_0, x_1 \mapsto \sqrt{-1} x_3 \mapsto x_1$.
We claim that the automorphism group is
$\C^\times (SO(U) \oplus SO(U')) \rtimes \Z/2$. To prove this it
suffices to show that any automorphism of $A$ in
$\C^\times \oplus \C^\times$ is diagonal, i.e., if $\psi \in \Aut(A)$
has the property that $\psi|_U$ and $\psi|_U'$ are scalar, then the
two scalars are equal.  Such an element must preserve the relation
$s_1$, which implies the needed result.

This yields the statement of the theorem for the case
$(\alpha, \beta, \gamma) = (\alpha, -1, 1)$ with $\alpha \neq \pm 1$:
although we have only computed the automorphism group of $A$, any
intertwiner $\psi: V \iso V$ which sends $A$ to some other $A'$ with
parameters $(\alpha', -1, 1)$ must also be of the above form, since
nothing depended on $\alpha$ (except that $\alpha \neq \pm 1$ so that
our statements about rank-two tensors are accurate).  

The case where $\alpha, \beta, \gamma \in \{\pm 1\}$ is trivial since
all of these cases are under the same orbit of $\Symm_3$ (and they
cannot be equivalent to any other example because their relations have
the largest subvarieties of rank-two tensors, or alternatively,
because we show in all other examples that this case is not
equivalent).

Thus, we have reduced the theorem to the nondegenerate case when
$\alpha, \beta,$ and $\gamma$ are all nonzero and \eqref{nondegwk} is
satisfied.  We will not make further mention of this assumption.  

Recall the Heisenberg group $H \cong \Aut(\omega)$ from the previous
section.  We will need the Stone-von Neumann theorem in our context
(we omit the proof, which is easy):
\begin{lemma} \label{svn} (\textbf{Stone-von Neumann theorem.})
  There is a unique irreducible representation of $H$ which sends
  elements $\zeta \in \mu_4$ to the corresponding scalar matrix
  $\zeta \cdot \Id$.  
\end{lemma}
Call this the \textbf{Heisenberg representation}.
Note that our given representation $V$ of $H$ is of this form.
\begin{notation} Let $\Aut(H,\mu_4)$ denote the subgroup of the
  automorphism group of $H$ which acts trivially on the center
  $\mu_4 < H$.  Similarly, let $\Inn(H, \mu_4) = \Inn(H)$ be the inner
  automorphisms, and $\Out(H, \mu_4)$ be $\Aut(H,\mu_4)$ modulo inner
  automorphisms.
\end{notation}
We know that a Sklyanin algebra is specified by a potential
$\omega \in V^{\otimes 4}$, up to a scalar multiple.  Now, let us fix
one such algebra $A_0$ with potential $\omega_0$. Then, $V$ naturally
has the structure of the unique irreducible Heisenberg representation
of Lemma \ref{svn}, given by any fixed isomorphism
$H \cong \Aut(\omega_0) \subset \Aut(V)$. Let
$\rho_0: H \rightarrow \Aut(V)$ be such a representation. 

So, we have fixed the data $(A_0, \omega_0, \rho_0)$. Now, given any
other algebra $A$ with potential $\omega \in V^{\otimes 4}$, it is
equipped with a Heisenberg representation
$\rho: H \rightarrow \Aut(V)$ which is unique up to precomposition
with an element of $\Aut(H,\mu_4)$. By Lemma \ref{svn} and Schur's
Lemma, there must be a unique up to scalar intertwiner
$\psi: V \iso V$ such that $\psi \rho_0(h) \psi^{-1} = \rho(g)$ for
all $h \in H$.  Hence, we obtain the vector
$\psi^{-1}(\omega) \in V^{\otimes 4}$.  This vector is uniquely
determined by $(A, \omega, \rho)$ up to scaling.

If we had picked a different potential $\omega$, this could also only
affect the vector $\psi^{-1}(\omega)$ by scaling.

If, instead of $\rho$, we had chosen $\rho' = \rho \circ \phi$ for
some element $\phi \in \Aut(H,\mu_4)$, then instead of
$\psi^{-1}(\omega) \in V^{\otimes 4}$, we would have obtained
$\psi_\phi^{-1} \psi^{-1}(\omega)$, where $\psi_\phi^{-1}: V \iso V$
is any intertwiner (unique up to scaling) between $\rho_0$ and
$\rho_0 \circ \phi$, i.e., such that
$\psi_\phi \rho_0(h) \psi_\phi^{-1} = \rho_0(\phi(h))$.

Note that, by Lemma \ref{svn}, we have a projective representation
$\Aut(H,\mu_4) \rightarrow PGL(V^{\otimes 4})$.
Thus, we have obtained a map from Sklyanin algebras to 
$\PP V^{\otimes 4} / \Aut(H, \mu_4)$.  In fact, we can do better: since
$\omega$ is fixed by the action of $\rho(H)$, $\psi^{-1}(\omega_0)$ is fixed
by the action of $\rho_0(H)$, and this is the same as the action of $\Inn(H, \mu_4)$ on $\PP V^{\otimes 4}$.  Hence, letting $U \subset V^{\otimes 4}$
be the subspace of fixed vectors under $\rho_0(H)$, we have a projective
representation of $\Out(H, \mu_4)$ on $U$, and have a map
\begin{equation} \label{ellsklymap}
\text{Four-dimensional Sklyanin algebras} \rightarrow \PP U / \Out(H, \mu_4).
\end{equation}
Furthermore, suppose we have $(A, \omega, \rho)$ as above, and another
 Sklyanin algebra $A' \cong A$, together with an isomorphism
$\theta: V \iso V$ carrying the relations of $A$ to the relations of
$A'$.  We may pick $\omega' = \theta(\omega)$ as our potential for
$A'$, and $\rho' := \theta \rho \theta^{-1}$ as our Heisenberg
representation $H \rightarrow \Aut(\omega')$.  Thus, using the
intertwiner $\psi' = \theta \circ \psi$, we see that the image of $A$
and $A'$ under \eqref{ellsklymap} is the same. Conversely, if we are
given $(A, \omega, \rho), (A', \omega', \rho'), \psi, \psi'$ such that
$\psi^{-1}(\omega) = (\psi')^{-1}(\omega')$, then
$\psi' \circ \psi^{-1}: V \iso V$ is an isomorphism carrying $\omega$
to $\omega'$, and hence induces an isomorphism between (the relations
of) $A$ and $A'$.

We thus obtain a canonical map (having fixed just $A_0$ and $\rho_0$):
\begin{equation}
\text{Isomorphism classes of four-dimensional Sklanin algebras} \hookrightarrow
\PP U / \Out(H, \mu_4).
\end{equation}

Next, we will describe the image of this map. Also, the reader will
probably recognize that $\Out(H, \mu_4) \cong \SL_2(\Z/4)$ and its
action on $\PP U$ is a version of the Weil representation, which we will
explain.

Let us define $U := (\Z/4)^{\oplus 2}$ and think of this as a free
rank-two $\Z/4$-module.
\begin{lemma}
The outer automorphism group $\Out(H,\mu_4)$ of $H$ fixing its center is 
$\SL_2(\Z/4)$.  We have the exact sequence
\begin{equation}
1 \rightarrow U \rightarrow \Aut(H, \mu_4) \rightarrow \SL_2(\Z/4) \rightarrow 1.
\end{equation}
\end{lemma}
Here, $\Aut(H, \mu_4)$ denotes the automorphism group of $H$ which
acts trivially on the center $\mu_4$.  Note that the size of
$\SL_2(\Z/4)$ is 48.
\begin{proof}
  It is clear that the inner automorphism group is $H/\mu_4 \cong U$.
  This acts by characters $U \rightarrow \mu_4$, fixing the center.
  Furthermore, the action of the outer automorphism group on $H$
  fixing the center descends to an action on $H/Z \cong U$, and this
  action must preserve commutators since $[H,H] \subset Z$. If we
  consider $x \wedge y \mapsto [x,y]$ for $x, y \in U \cong H/Z$ to be
  a volume form, then we obtain an embedding
  $\Out(H, \mu_4) \hookrightarrow \SL_2(\Z/4)$.  We have to show this is
  surjective.  If $X, Y$ are lifts of generators of $U$ to $H$, they
  have order $8$, and it follows that the same is true for $X^a Y^b$
  whenever at least one of $a, b$ is odd.  As a result, we see that,
  for any two elements $X', Y' \in H$ such that $[X',Y'] = \mu_4$, the
  map $X \mapsto X', Y \mapsto Y'$ must yield an automorphism of $H$
  fixing $\mu_4$.
\end{proof}

As a consequence, the action of $\Out(H, \mu_4)$ on $\PP U$ is a 
projective representation
of $\SL_2(\Z/4)$, which we will call the Weil representation on $U$.

Let $\cS_0 \subset \cS$ be the subset of tuples satisfying the assumptions
of Theorem \ref{sszthm}.
Next, we will describe explicitly the map $\cS_0 \rightarrow \PP U / \SL_2(\Z/4)$ and show that its kernel is $\Symm_3$.  More precisely, we show that this
map factors as follows.  Let $K \subset \SL_2(\Z/4)$ be the kernel of the
canonical surjection $\SL_2(\Z/4) \twoheadrightarrow \SL_2(\Z/2)$ (note
that $K \cong (\Z/2)^{\times 3}$). Then, we prove the following
\begin{claim}
The map $(\alpha, \beta, \gamma) \mapsto \omega$ given by
\eqref{sklypot} factors as follows:
\begin{equation}
\cS_0 \hookrightarrow \PP U / K \twoheadrightarrow \PP U / \SL_2(\Z/4).
\end{equation}
Moreover, using a natural isomorphism $\Symm_3 \cong \SL_2(\Z/2)$, the
action of $\Symm_3$ on $\cS$ is identified with the action of
$\SL_2(\Z/2)$ on $\PP U / (\Z/2)^{\times 3}$.
\end{claim}
The theorem follows immediately from the claim.

To prove the claim, we recall from \cite[\S 2]{ssz} explicit formulas
for $\rho_0(X), \rho_0(Y)$.  Let
$\theta_0, \theta_1, \theta_2, \theta_3 \in \C^\times$ be numbers such that
\begin{equation}\label{thetdfn}
  \alpha_0 = \bigl( \frac{\theta_0 \theta_1}{\theta_2 \theta_3} \bigr)^2, \quad \beta_0 = - \bigl( \frac{\theta_0 \theta_2}{\theta_1 \theta_3} \bigr)^2, \quad \gamma_0 = - \bigl( \frac{\theta_0 \theta_3}{\theta_1 \theta_2} \bigr)^2.
\end{equation}
(The numbers $\theta_i$ are in fact Jacobi's four theta-functions
associated with an elliptic curve valued at a point of that curve,
which may be used to give a geometric definition of $A_0$.  We will not
need this fact.)  Fix $i = \sqrt{-1} \in \C$. We have: 
\begin{equation}\label{rho0}
\rho_0(X) = \begin{pmatrix} 0 & 0 & 0 & i \frac{\theta_3}{\theta_0} \\
0 & 0 & -i \frac{\theta_2}{\theta_1} & 0 \\ 0 & i \frac{\theta_1}{\theta_2} & 0 & 0 \\ i \frac{\theta_0}{\theta_3} & 0 & 0 & 0 
\end{pmatrix}, \quad \rho_0(Y) = \begin{pmatrix} 0 & 0 & -i \frac{\theta_2}{\theta_0} & 0 \\ 0 & 0 & 0 & -\frac{\theta_3}{\theta_1} \\ i \frac{\theta_0}{\theta_2} & 0 & 0 & 0 \\ 0 & \frac{\theta_1}{\theta_2} & 0 & 0 \end{pmatrix}.
\end{equation}
Then, if $(\alpha, \beta, \gamma) \in \cS_0$, for any choice of
$\theta_0', \theta_1', \theta_2', \theta_3'$ satisfying the version of
\eqref{thetdfn} for $(\alpha, \beta, \gamma)$, we may define the
representation $\rho$ using \eqref{rho0} with primed thetas.  It is
easy to see that an intertwiner $\psi$ carrying $\rho_0$ to $\rho$ is
given by
\begin{equation}
  \psi = \begin{pmatrix} \theta_0 / \theta'_0 & 0 & 0 & 0 \\ 0 & \theta_1 / \theta'_1 & 0 & 0 \\ 0 & 0 & \theta_2/\theta'_2 & 0 \\ 0 & 0 & 0 & \theta_3/\theta'_3 \end{pmatrix}.
\end{equation}
As a consequence, we obtain a vector $\psi^{-1}(\omega)$ in $U$. However,
the construction involved a choice of the $\theta_j'$, so it is not yet
well-defined.  First, nothing is affected by multiplying all the $\theta_j'$ by
the same scalar, since everything only involves ratios of the same number
of the thetas. So let us assume that $\theta_0' = 1$.  Any other choice
of $\theta_1', \theta_2', \theta_3'$ must differ by a transformation
\begin{equation} \label{vemat}
\begin{pmatrix} 1 \\ \theta_1' \\ \theta_2' \\ \theta_3' \end{pmatrix} \mapsto \begin{pmatrix} 1 & 0 & 0 & 0 \\ 0 & \varepsilon_1 & 0 & 0 \\ 0 & 0 & \varepsilon_2 & 0 \\ 0 & 0 & 0 & \varepsilon_3 \end{pmatrix} \begin{pmatrix} 1 \\ \theta_1' \\ \theta_2' \\ \theta_3' \end{pmatrix},
\end{equation}
where $\varepsilon_i \in \mu_4$, and
$\varepsilon_1 \varepsilon_2 \varepsilon_3 = \pm 1$.  First of all, in
the case that $\varepsilon_j \in \{\pm 1\}$ for all $j$ and
$\varepsilon_1 \varepsilon_2 \varepsilon_3 = 1$, then the matrix
$\begin{pmatrix} 1 & 0 & 0 & 0 \\ 0 & \varepsilon_1 & 0 & 0 \\ 0 & 0 &
  \varepsilon_2 & 0 \\ 0 & 0 & 0 & \varepsilon_3 \end{pmatrix}$
is already in $\rho_0(H)$ (and $\rho(H)$, so it will not affect
$\psi^{-1}(\omega)$.  Factoring the group of $\varepsilon$-matrices
\eqref{vemat} by this subgroup leaves a group isomorphic to
$(\Z/2)^{\times 3}$.  Conjugating $\rho_0$ by the action of this group
is easily verified to send $\rho_0$ to $\rho_0 \circ K$, where
$K \subset \SL_2(\Z/4) \cong \Out(H)$ is the kernel of
$\SL_2(\Z/4) \twoheadrightarrow \SL_2(\Z/2)$.  After all, given any
$h \in H$, the elements $i^k\rho_0(h X^{2\ell} Y^{2m})$ for
$k, \ell, m \in \Z$ are exactly those that differ from $h$ by a
diagonal matrix.  Hence, we obtain a well-defined map from tuples
$(\alpha, \beta, \gamma) \in \cS_0$ to $\PP U / K$.

We claim that the resulting map $\cS_0 \rightarrow \PP U  / K$ is injective.
To see this, note that, since $\psi^{-1}$ is diagonal, we may recover
$\alpha$ from $\psi^{-1}(\omega)$ as follows: Write $\psi^{-1}(\omega)$
as a linear combination of terms of the form 
\begin{equation}
[x_i, x_j] \{x_k, x_\ell\}, \quad [x_i, x_j] \{x_k, x_\ell\}, \quad \{x_i, x_j\} [x_k, x_\ell], \quad \{x_i, x_j\} \{x_k, x_\ell\},
\end{equation}
where, as before, $\{x,y\} := xy+yx$ is the anticommutator.  We see that
\begin{equation}
\frac{\text{Coefficient in $\psi^{-1}(\omega)$ of } \{x_0, x_1\} \{x_2, x_3\}}{
\text{Coefficient in $\psi^{-1}(\omega)$ of } [x_0, x_1] [x_2, x_3]} = \alpha.
\end{equation}
This does not depend on rescaling $\omega$.  Similarly, we may recover
$\beta, \gamma$ from $\psi^{-1}(\omega)$.  This proves injectivity.

It remains only to show that the action of $\SL_2(\Z/2)$ is identified
with the action given in the theorem of $\Symm_3$ under an isomorphism
$\SL_2(\Z/2) \cong \Symm_3$.  Since $\Symm_3$ clearly acts by
automorphisms and faithfully so except at two points, this must be
true, but we give a direct argument.  The intertwining action
$\Aut(H,\mu_4) \rightarrow PGL(V)$ is easily seen to be given by
matrices which are products of diagonal matrices with permutation
matrices (just like all the formulas above).  

Thus, we have a map $P: \Aut(H,\mu_4) \rightarrow \Symm_4$ given by
modding by diagonal matrices.  On the other hand, we see that
$\rho(X^2), \rho(Y^2)$ are diagonal matrices, so that $P|_{\Inn(H)}$
descends to a map $Q: (\Z/2 \times \Z/2) \rightarrow \Symm_4$ under the
quotient $\Inn(H) \cong (\Z/4 \times \Z/4) \onto (\Z/2 \times \Z/2)$.
This map $Q$ is an isomorphism onto the normal subgroup
$\{(ab)(cd)\} \cong (\Z/2 \times \Z/2) \subset \Symm_4$, as is clear
from \eqref{rho0}.

As a result, the map $P$ itself descends to a map
$\overline{P}: \SL_2(\Z/2) \rightarrow \Symm_4 / (\Z/2 \times \Z/2) \cong
\Symm_3$.
This is the isomorphism sending an element of $\SL_2(\Z/2)$ to the
permutation induced on the three nonzero elements of
$\Z/2 \times \Z/2$: after all, for $\phi \in \SL_2(\Z/2)$ and
$w \in \Z/2 \times \Z/2$, we have
$Q(\phi(w)) = \overline{P}(\phi) Q(w)$.   This completes the proof.



\bibliographystyle{amsplain}
\def\cprime{$'$}
\providecommand{\bysame}{\leavevmode\hbox to3em{\hrulefill}\thinspace}
\providecommand{\MR}{\relax\ifhmode\unskip\space\fi MR }
\providecommand{\MRhref}[2]{%
  \href{http://www.ams.org/mathscinet-getitem?mr=#1}{#2}
}
\providecommand{\href}[2]{#2}

\end{document}